\documentclass[11pt,oneside]{ip-journal}
\usepackage{geometry}
\usepackage{epstopdf}
\usepackage{mathtools}
\usepackage[cmtip,all]{xy}

\usepackage{mathrsfs}
\usepackage{amscd}
\usepackage{amsmath}
\usepackage{amssymb}
\usepackage{amsthm}
\usepackage{float}
\usepackage[dvips]{graphicx}
\usepackage{color}
\usepackage{tikz} 
\usepackage{enumerate}
\usepackage{appendix}
\usetikzlibrary{matrix,arrows}

\usepackage{array}
\usepackage{amscd}
\usepackage[all]{xy}

\DeclareFontFamily{U}{rsfs}{%
\skewchar\font127}
\DeclareFontShape{U}{rsfs}{m}{n}{%
<-6>rsfs5<6-8.5>rsfs7<8.5->rsfs10}{}
\DeclareSymbolFont{rsfs}{U}{rsfs}{m}{n}
\DeclareSymbolFontAlphabet
{\mathrsfs}{rsfs}
\DeclareRobustCommand*\rsfs{%
\@fontswitch\relax\mathrsfs}

\DeclareFontFamily{U}{rsfs}{%
\skewchar\font127}
\DeclareFontShape{U}{rsfs}{m}{n}{%
<-6>rsfs5<6-8.5>rsfs7<8.5->rsfs10}{}
\DeclareSymbolFont{rsfs}{U}{rsfs}{m}{n}
\DeclareSymbolFontAlphabet
{\mathrsfs}{rsfs}
\DeclareRobustCommand*\rsfs{%
\@fontswitch\relax\mathrsfs}

\theoremstyle{plain}
\newtheorem{theorem}{Theorem}
\newtheorem{thm}{Theorem}[section]
\newtheorem{prop}[thm]{Proposition}
\newtheorem{lem}[thm]{Lemma}

\newtheorem{defi}[thm]{Definition}
\newtheorem{rmk}[thm]{Remark}
\newtheorem{cor}[thm]{Corollary}

\newtheorem*{prop*}{Proposition}
\newtheorem*{notn}{Notation}

\newtheorem{prop-defi}[thm]{Proposition-Definition}
\newtheorem{thm-defi}[thm]{Theorem-Definition}
\newtheorem{lem-defi}[thm]{Lemma-Definition}

\newcommand{\C}{\mathbb{C}}
\newcommand{\ZZ}{\mathbb{Z}}
\newcommand{\Z}{\mathcal{Z}}
\newcommand{\I}{\mathcal{I}}
\renewcommand{\div}{\operatorname{div}}
\renewcommand{\det}{\operatorname{det}}
\newcommand{\pts}{\operatorname{pts}}
\newcommand{\pic}{\operatorname{Pic}}
\renewcommand{\O}{\mathcal{O}}
\newcommand{\n}{{\boldsymbol{n}}}

\newcommand{\betab}{{\boldsymbol{\beta}}}
\renewcommand{\L}{{\boldsymbol{L}}}
\renewcommand{\S}[1]{S^{[#1]}}

\renewcommand{\i}[1]{\I^{[#1]}}
\newcommand{\fI}[1]{\mathfrak{I}^{[#1]}}
\newcommand{\id}{\operatorname{id}}
\newcommand{\pr}{\operatorname{pr}}
\newcommand{\LL}{\mathbb{L}^{\bullet}}

\newcommand{\Fb}{F^{\bullet}}
\newcommand{\Ab}{A^{\bullet}}
\newcommand{\Bb}{B^{\bullet}}
\newcommand{\Cb}{C^{\bullet}}
\newcommand{\Eb}{E^{\bullet}}
\newcommand{\Gb}{G^{\bullet}}
\newcommand{\Fbv}{F^{\bullet \vee}}
\newcommand{\Fbr}{F^{\bullet}_{\operatorname{rel}}}
\newcommand{\Fbrv}{F^{\bullet \vee}_{\operatorname{rel}}}
\newcommand{\dR}{\mathbf{R}}
\newcommand{\dL}{\mathbf{L}}
\renewcommand{\hom}{\mathcal{H}om}
\newcommand{\ext}{\mathcal{E}xt}
\newcommand{\Hom}{\operatorname{Hom}}
\newcommand{\Ext}{\operatorname{Ext}}
\newcommand{\cone}{\operatorname{Cone}}

\newcommand{\tr}{\operatorname{tr}}
\newcommand{\p}{\operatorname{p}}
\newcommand{\q}{\operatorname{q}}
\newcommand{\coker}{\operatorname{coker}}
\newcommand{\at}{\operatorname{At}}
\newcommand{\Tb}{\overline{T}}

\newcommand{\rank}{\operatorname{Rank}}
\newcommand{\mM}{\mathcal{M}}

\newcommand{\vir}{\operatorname{vir}}
\newcommand{\dT}{\mathbf{T}}

\newcommand{\ct}[1]{(\C^2)^{[#1],\dT}}
\newcommand{\ST}[1]{S^{[#1],\dT}}

\newcommand{\SD}[1]{(S/D)^{[#1]}}

\newcommand{\cT}{\mathcal{T}}
\newcommand{\cZ}{\mathsf Z}
\newcommand{\fC}{\mathfrak{C}}
\renewcommand{\exp}{\operatorname{exp}}
\newcommand{\str}{\operatorname{str}}
\renewcommand{\c}{\mathsf{c}}

\newcommand{\F}{\mathcal{F}}
\newcommand{\FF}{\mathbb{F}}
\newcommand{\IIb}{\mathbb{I}^\bullet}

\newcommand{\T}{\mathcal{T}}

\newcommand{\D}{\mathcal{D}}
\newcommand{\nest}{\operatorname{nest}}

\newcommand{\red}{\operatorname{red}}
\newcommand{\rel}{\operatorname{rel}}
\newcommand{\ob}{\operatorname{ob}}
\newcommand{\cS}{\mathcal{S}}
\newcommand{\cP}{\mathcal{P}}
\newcommand{\sN}{\mathsf{N}}

\newcommand{\sF}{\mathsf{F}}
\newcommand{\sK}{\mathsf{K}}
\newcommand{\sT}{\mathsf{T}}
\newcommand{\sE}{\mathsf{E}}

\newcommand{\kk}[1]{\sK^{[#1]}}

\newcommand{\PP}{\mathbb{P}}
\newcommand{\spec}{\operatorname{Spec}}

\title[Nested Hilbert schemes on Surfaces]{Nested Hilbert schemes on surfaces:\\ Virtual fundamental class}

\author{Amin Gholampour and Artan Sheshmani and Shing-Tung Yau}

\begin{document}

\maketitle

\begin{abstract}
We construct virtual fundamental classes on nested Hilbert schemes of points and curves in complex nonsingular projective surfaces. These classes recover the virtual classes of Seiberg-Witten theory as well as the (reduced) stable theory, and play a crucial role in the reduced Donaldson-Thomas theory of local-surface-threefolds that we study in \cite{GSY17b}. We show that certain integrals against the virtual fundamental classes of punctual nested Hilbert schemes are expressed as integrals over the products of the Hilbert scheme of points. We are able to find explicit formulas for some of these integrals by relating them to Carlsson-Okounkov's vertex operator formulas.   
\end{abstract}

\setcounter{tocdepth}{3}
\tableofcontents

\section{Introduction} Hilbert scheme of points on a nonsingular surface $S$ have been vastly studied. They are nonsingular varieties with rich geometric structures some of which have applications in physics (see \cite{N99} for a survey).  We are mainly interested in the enumerative geometry of Hilbert schemes of points \cite{G90, L99, CO12, GS16}. This has applications in curve counting problems on $S$  \cite{LT14, R17}. The first two authors of this paper have studied the relation of some of these enumerative problems to the Donaldson-Thomas theory of 2-dimensional sheaves in threefolds and to S-duality conjectures \cite{GS13}.
In contrast, Hilbert scheme of curves on $S$ can be badly behaved and singular. They were studied in detail by D\"{u}rr-Kabanov-Okonek  \cite{DKO07} in the context of Poincar\'e invariants (algebraic Seiberg-Witten invariants \cite{CK13}). More recently, the stable pair invariants of surfaces have been employed in the context of curve counting problems \cite{PT10, MPT10, KT14, KST11}. The moduli space of stable pairs on $S$ is identified with the Hilbert scheme of points on curves on $S$, and so is a nested Hilbert scheme on $S$. This paper studies general nested Hilbert schemes of points and curves on $S$. The geometry of nested Hilbert scheme of points was studied in \cite{C98}. We construct a reduced and a non-reduced perfect obstruction theories for the nested Hilbert scheme of curves and points on $S$. The reduced one is an extension of the reduced perfect obstruction theory of stable pairs worked out in \cite{KT14}, and its construction closely follows their technique. Our main application of the non-reduced perfect obstruction theory is in the study of (reduced!)  Donaldson-Thomas theory of local surfaces that is carried out in  \cite{GSY17b}. The non-reduced perfect obstruction theory appears in the study of Vafa-Witten theory of $S$ \cite{TT17}, and also in the work of  A. Negut  \cite{N17}.

\subsection{Nested Hilbert schemes on surfaces}
Suppose that $r\in \ZZ_{>0}$, $\n:=n_{1},n_{2},\dots,n_{r}\in \ZZ_{\ge 0}$, and $\betab:=\beta_1,\dots,\beta_{r-1}\in H^2(S,\ZZ)$. We denote the corresponding nested Hilbert scheme $\S{\n}_\betab$, whose underlying set of closed points consists of tuples of subschemes of $S$ $$(Z_1,Z_2,\dots,Z_r),\quad (C_1,\dots,C_{r-1})$$ where $Z_i$ is a 0-dimensional of length $n_i$, and $C_i$ a divisor with $[C_i]=\beta_i$, and  for any $i<r$,  \begin{equation}\label{incl} I_{Z_{i}}(-C_i)\subseteq I_{Z_{i+1}},\end{equation} where $I_{Z_i}\subseteq \O_S$ is the ideal sheaf of $Z_i$. \footnote{Taking double dual shows that  an inclusion of ideals $I_{Z_i}(-D_i)\subseteq I_{Z_{i+1}}(-D_{i+1})\subset \O_S$ is equivalent to \eqref{incl} with $C_i=D_i-D_{i+1}$. So we are not losing information by starting with $r-1$ effective divisors $C_1,\dots, C_{r-1}$ instead of $r$ of them.}   
In other words, a closed point of $\S{\n}_\betab$ corresponds to a chain of subschemes of $S$ given by the ideals
$$I_{Z_1}(-C_1-\cdots- C_{r-1})\subseteq I_{Z_2}(-C_2-\cdots -C_{r-1})\subseteq \cdots \subseteq I_{Z_r}\subseteq \O_S.$$

To define invariants of $S$ arising from $\S{\n}_\betab$  (see Section \ref{invs}), we construct a  virtual fundamental class $[\S{\n}_\betab]^{\vir}$. More precisely, we construct a natural (non-reduced) perfect obstruction theory over  $\S{\n}_\betab$. This is done by studying the deformation/obstruction theory of the maps of coherent sheaves given by the natural inclusions \eqref{incl} following Illusie \cite{Ill} and Joyce-Song \cite{JS12} \footnote{Another approach would be to study the deformation/obstruction theory of quotients $I_{Z_{i+1}}/I_{Z_i}(-C_i)$ following the work of Gillam \cite{G11}, which in turn also uses a great deal of \cite{Ill}.  We give a brief sketch of this approach in Subsection \ref{gillam}.}.  As we will see, this in particular provides a uniform way of studying all known perfect obstruction theories on the Hilbert schemes of points and curves, as well as the stable pair moduli spaces on $S$. The main result of the paper is (Propositions \ref{abs-r}, \ref{abs} and Corollary \ref{virclass}):
\begin{theorem} \label{thm1}
Let $S$ be a nonsingular projective surface over $\C$ and $\omega_S$ be its canonical bundle.The nested Hilbert scheme $\S{\n}_\betab$ with $r\ge 2$ carries a natural perfect obstruction theory with the virtual fundamental class $$[\S{\n}_\betab]^{\vir} \in A_d( \S{\n}_\betab), \quad \quad  d=n_1+n_r+\frac{1}{2}\sum_{i=1}^{r-1}\beta_i\cdot(\beta_i-c_1(\omega_S)).$$ 
\end{theorem}


For $r\ge 2$ and $\beta_i=0$, $\S{\n}:=\S{\n}_{(0,\dots,0)}$ is the nested Hilbert scheme of points on $S$ parameterizing flags of 0-dimensional subschemes $Z_r\subseteq \dots \subseteq Z_2\subseteq Z_1\subset S$.  $\S{\n}$ is in general singular of actual dimension $2n_1$. 

We are specifically interested in the case $r=2$ in this paper: $\S{\n}_\betab=\S{n_1,n_2}_\beta$ for some $\beta\in H^2(S,\ZZ)$. The cases $S_\beta:=\S{0,0}_\beta$ and $\S{0,n_2}_\beta$ (when $\beta\neq 0$) are respectively isomorphic to the Hilbert scheme of divisors and the moduli  space of stable pairs in classes $\beta$.  The comparison of the virtual class of Theorem \ref{thm1} to other known virtual classes is carried out in  Proposition \ref{thm1.2}.

In certain cases, we construct a \emph{reduced} virtual fundamental class  $$[\S{n_1,n_2}_\beta]^{\vir}_{\red} \in A_{n_1+n_2+\frac{1}{2}\beta\cdot(\beta-c_1(\omega_S))+p_g(S)}(\S{n_1,n_2}_\beta)$$ by reducing the perfect obstruction theory in Theorem \ref{thm1} (Propositions \ref{vanish}, \ref{reduced}). The reduced virtual fundamental class $[\S{0,n_2}_\beta]^{\vir}_{\red}$ match with the reduced virtual fundamental class of stable pair theory constructed in \cite{KT14} (Proposition \ref{thm1.2}). 

\subsection{Punctual nested Hilbert schemes} Sections \ref{sec:nestpoints} and \ref{sec:co} are devoted to study of the nested Hilbert schemes of points $$\S{n_1\ge n_2}:=\S{n_1,n_2}_{\beta=0}$$ in more detail. Let  $$\iota: \S{n_1\ge n_2}\hookrightarrow \S{n_1}\times \S{n_2}$$ be the natural inclusion. If $S$ is toric with the torus $\dT$ and the fixed set $S^\dT$, in Section \ref{sec:toric} we provide a purely combinatorial formula for computing $[\S{n_1\ge n_2}]^{\vir}$ by torus localization along the lines of \cite{MNOP06}. Let $d$ be a positive integer, by a partition $\mu$ of $d$, denoted by  $\mu \vdash d$, we mean a finite sequence of positive integers  
$$\mu=(\mu_1\ge \mu_2 \ge \mu_3\ge \dots) \quad \text{such that}\quad  d=\sum_i \mu_i.$$ The number of $\mu_i$'s is called the length of the partition $\mu$, and is denoted by $\ell(\mu)$. If $\mu'\vdash d'$ with $d'\le d$, we say $\mu' \subseteq \mu$ if $\ell(\mu')\le \ell(\mu)$ and  $\mu'_i \le \mu_i$ for all $1\le i\le \ell(\mu')$. 

\begin{theorem} \label{thm1.7}
For a toric nonsingular surface $S$ the $\dT$-fixed set of $\S{n_1,n_2}$ is isolated and in bijective correspondence by the tuples of nested partitions: $$\left\{(\mu'_{P} \subseteq \mu_{P})_{P \in S^\dT}\mid  \mu'_{P} \vdash d'_P, \quad \mu_{P} \vdash d_P, \quad n_2=\sum_P d'_P, \quad n_1=\sum_P d_P \right\}.$$
Moreover, the $\dT$-character of the virtual tangent bundle $\T^{\vir}$ of $\S{n_1\ge n_2}$ at the fixed point $Q=(\mu'_{P} \subseteq \mu_{P})_{P\in S^{\dT}}$ is given by  $$ \tr_{\cT^{\vir}_{Q}}(t_1,t_2)=\sum_{P\in S^{\dT}} \mathsf{V}_P,$$ where $t_1, t_2$ are the torus characters and $\mathsf{V}_P$ is a Laurent polynomial in $t_1, t_2$ that is completely determined by the partitions $\mu'_{P}$ and $\mu_{P}$ and is given by the right hand side of formula \eqref{virtan}. 
\end{theorem}
The bijection in Theorem \ref{thm1.7} is deduced from the standard bijection between the set of the fixed points of punctual Hilbert schemes on the affine plane and the set of monomial ideals of finite colegnths in the polynomial rings in two variables. 

Let $\I_1, \I_2$ be the universal ideal sheaves on $S\times \S{n_1}\times \S{n_2}$, and let $\pi$ be the projection to the last two factors $\S{n_1}\times \S{n_2}$ and $p$ be the projection to $S$. Following \cite{CO12}, for any line bundle $M$ on $S$,  we define a $K$-theory class on $\S{n_1}\times \S{n_2}$  by $$\sE^{n_1,n_2}_M:=\sum_{i=0}^2(-1)^i\big(H^i(S,M)\otimes\O- \ext^i_{\pi}(\I_1,\I_2\otimes p^*M)\big).$$ If $M$ is trivial we drop it from the notation (see Definition \ref{virbdl}). When $S$ is toric, by torus localization, we can express $[\S{n_1\ge n_2}]^{\vir}$ in terms of the fundamental class of the product of Hilbert schemes $\S{n_1}\times \S{n_2}$ (Proposition \ref{nestprodfano}):

\begin{theorem} \label{thm2} If $S$ is a nonsingular projective toric surface, then, $$\iota_*[\S{n_1\ge n_2}]^{\vir}=[\S{n_1}\times \S{n_2}]\cap c_{n_1+n_2}(\sE^{n_1,n_2}).$$ 
\end{theorem}

Theorem \ref{thm2} holds in particular for $S=\mathbb{P}^2,\; \mathbb{P}^1\times \mathbb{P}^1$, which are the generators of the cobordism ring of nonsingular projective surfaces. Similar formulas was worked by A. Negut and others (see \cite{N12, N17} and the references within). We use a refinement of this fact together with a degeneration formula developed for $[\S{n_1\ge n_2}]^{\vir}$ (Proposition \ref{degen}) to prove that for \emph{any} nonsingular projective surface $S$, certain integrals against $[\S{n_1\ge n_2}]^{\vir}$ can be expressed as integrals against $\S{n_1}\times \S{n_2}$  (Corollary \ref{cor:znzp}, Proposition \ref{genznzp}). A generalization of such formulas have been recently worked out in \cite{GT17, GT19} using degeneracy loci techniques.

This last result (Proposition \ref{genznzp}) is used in \cite{GSY17b} to express some of the reduced localized DT invariants of $S$ as sums of integrals over the product of Hilbert schemes of points on $S$, when $S$ is one of five types generic complete intersections $$(5)\subset \mathbb{P}^3, \; (3,3)\subset \mathbb{P}^4,\; (4,2)\subset \mathbb{P}^4, \; (3,2,2)\subset \mathbb{P}^5,\; (2,2,2,2)\subset \mathbb{P}^6.$$ Such integrals also have applications in evaluating Vafa-Witten invariants recently defined in \cite{TT17}.

The operators $$\int_{\S{n_1}\times \S{n_2}}- \cup c_{n_1+n_2}(\sE^{n_1,n_2}_M)$$ were studied by Carlsson-Okounkov in \cite{CO12}. They expressed these operators in terms of explicit vertex operators. Using this, we prove the following explicit formula (Proposition \ref{coformul}):
\begin{theorem} \label{thm4}  Let $S$ be a nonsingular projective surface, $\omega_S$ be its canonical bundle, and $K_S=c_1(\omega_S)$. Then,
\begin{align*}\sum_{n_1\ge n_2\ge 0}(-1)^{n_1+n_2}&\int_{[\S{n_1\ge n_2}]^{\vir}} \iota^*c(\sE^{n_1,n_2}_M)q_{1}^{n_1}q_2^{n_2}=\\&\prod_{n> 0}\left(1- q_2^{n-1}q_1^n\right)^{\langle K_S, K_S-M \rangle}\left(1- q_1^nq_2^{n}\right)^{\langle K_S-M, M \rangle-e(S)},
\end{align*}
where $\langle-,-\rangle$ is the Poincar\'e paring on $S$.
\end{theorem}
This is the only situation that we have been able to find a closed formula for the complete generating series of invariants. It would be interesting to seek similar formulas for the more involved integrals over nested Hilbert schemed that appear in Vafa-Witten theory or reduced local DT theory of $S$.

\section*{Aknowledgement}
We are grateful to Richard Thomas for providing us with many valuable comments. We would like to thank Eric Carlsson, Andrei Negut, Davesh Maulik, Hiraku Nakajima, Takur\={o} Mochizuki, Alexey Bondal and Mikhail Kapranov for useful discussions. 

A. G. was partially supported by NSF grant DMS-1406788. A. S. was partially supported by World Premier 
International Research Center Initiative (WPI initiative), MEXT, Japan, as well as NSF DMS-1607871, NSF DMS-1306313 and Laboratory of Mirror Symmetry NRU HSE, RF Government grant, ag. No 14.641.31.0001. S.-T. Y. was partially supported by NSF DMS-0804454, NSF PHY-1306313, and Simons 38558. A. S. would like to further sincerely thank the center for Quantum Geometry of Moduli Spaces at Aarhus University, the Center for Mathematical Sciences and Applications at Harvard University and the Laboratory of Mirror Symmetry in Higher School of Economics, Russian federation, for the great help and support. 

\subsection{Notation and conventions} \label{notcon}
\begin{enumerate}
\item We will use the symbol $\mathbb L^{\bullet}$ for the (full) cotangent complex, and  
 $\mathbb{L}^{\bullet,\text{gr}}$ for the cotangent complex for the sheaves of graded algebras. Our sheaves of graded algebras are always of the form $A_0\oplus A_1$ where $A_i$ is in degree $i$. If $M$ is a complex of graded $(A_0\oplus A_1)$-modules then $k^i(M)$ takes the degree $i$ part in the grading. If $$A_0\oplus A_1\to B_0\oplus B_1, \qquad A_0\oplus A_1\xrightarrow{\id \oplus s} A_0\oplus C_1$$ are graded homomorphisms of sheaves of graded algebras in which $s$ is injective then, we will use the following isomorphisms proven in \cite[IV.2.2.4), IV.2.2.5, IV.3.2.10]{Ill} 
$$k^0\left(\mathbb{L}^{\bullet,\text{gr}}_{(B_0\oplus B_1)/(A_0\oplus A_1)}\right)\cong \LL_{B_0/A_0}, \quad k^1\left(\mathbb{L}^{\bullet,\text{gr}}_{(A_0\oplus C_1)/(A_0\oplus A_1)}\right)\cong \coker(s).$$
Associated, to graded homomorphisms of graded sheaves of  algebras  $$A_0\oplus A_1\to B_0\oplus B_1\to C_0\oplus C_1$$  is a natural exact triangle $$\mathbb{L}^{\bullet,\text{gr}}_{(B_0\oplus B_1)/(A_0\oplus A_1)}\otimes_{(B_0\oplus B_1)}(C_0\oplus C_1) \to \mathbb{L}^{\bullet,\text{gr}}_{(C_0\oplus C_1)/(A_0\oplus A_1)}\to\mathbb{L}^{\bullet,\text{gr}}_{(C_0\oplus C_1)/(B_0\oplus B_1)}$$ 
that is referred to as the \emph{transitivity triangle} \cite[IV.2.3]{Ill}.
\item We will denote the universal ideal sheaves of $\S{m}_{\beta}$, $\S{m}$, and $S_\beta$ respectively by $\i{m}_{-\beta}$, $\i{m}$, and $\I_{-\beta}$, and the corresponding universal subschemes respectively by $\Z^{[m]}_\beta$, $\Z^{[m]}$, and $\Z_\beta$. We will also write $\i{m}_{\beta}$ for $\i{m}\otimes \O(\Z_\beta)$. Using the universal property of the Hilbert scheme, it can be seen that $\i{m}_{-\beta}\cong \i{m}\otimes\O(-\Z_\beta)$. 
\item Let $\pi:S\times \S{\n}_\betab\to \S{\n}_\betab$ be the projection, we denote the derived functor $\dR\pi_*\dR\hom$ by $\dR\hom_\pi$ and its $i$-th cohomology sheaf by $\ext^i_\pi$. $\pi$ is a smooth morphism of relative dimension 2 and hence by Grothendieck-Verdier duality $\pi^!(-):=\pi^*(-)\otimes \omega_{\pi}[2]$ is a right adjoint of $\dR \pi_*$.

\item Throughout the paper, we slightly abuse notation and suppress many of the symbols $f^*$ and $f^{-1}$ for the pullback of sheaves on $Y$ via a given morphism $f:X\to Y$ of schemes. This makes most of the formulas notationally lighter and hence more readable.

\item For any line bundle $L$ on $S$ we define $L^D:=L^{-1}\otimes \omega_S$. Similarly, for any class $\beta\in H^2(S,\ZZ)$ we define $\beta^D:=K_S-\beta$.
\end{enumerate}

\section{Nested Hilbert schemes on surfaces} \label{sec:general}
Let $S$ be a nonsingular projective surface over $\C$. We denote the canonical line bundle on $S$ by $\omega_S$ and $K_S:=c_1(\omega_S)$. 
For any nonnegative integer $m$ and effective curve class $\beta \in H^2(S,\ZZ)$, we denote by $\S{m}_\beta$ the Hilbert scheme of 1-dimensional subschemes $Z\subset S$ such that $$[Z]=\beta,\quad c_2(I_Z)=m.$$ If $\beta=0$ we drop it from the notation and denote by $\S{m}$ the Hilbert scheme of $m$ points on $S$.  Similarly, in the case $m=0$ but $\beta\neq 0$ we drop $m$ from the notation and use $S_\beta$ to denote the Hilbert scheme of curves in class $\beta$. There are natural morphisms $$\div:\S{m}_\beta \to S_\beta, \quad \det:\S{m}_\beta \to \pic(S), \quad \pts:\S{m}_\beta \to \S{m},$$ where $\div$ sends a 1-dimensional subscheme $Z\subset S$ to its underlying divisor on $S$,  $\det(Z):=\O(\div(Z)),$ and $\pts(Z)$ is the 0-dimensional subscheme of $S$ defined by the ideal $I_Z(\div(Z))$. In fact $\div(-)$ is well-behaved with respect to basechange (see \cite{F69, KM77}) and sends a flat family of 1-dimensional subschemas of $S$ to a flat family of divisors on $S$, and so using the universal property of the Hilbert schemes, these maps are morphisms of schemes.  There is a natural isomorphism of schemes $\S{m}_\beta\cong\S{m}\times S_\beta$ defined by mapping $Z\mapsto ( \pts(Z), \div(Z))$. Under this, we have the following relation among the universal ideal sheaves: $\i{m}_{-\beta}\cong \i{m}\boxtimes \O(-\Z_\beta)$.

It is well known that $\S{m}$ is a nonsingular variety of dimension $2m$. The tangent bundle of $\S{m}$ is identified with 
\begin{equation} \label{tanhilb} T_{\S{m}}\cong \hom_\pi\big(\i{m},\O_{\Z^{[m]}}\big)\cong \dR\hom_{\pi}\big(\i{m},\i{m}\big)_0[1]\cong \ext^1_\pi\big(\i{m},\i{m}\big)_0,\end{equation} where the index 0 indicates the trace-free part.

The main object of study in this paper is the following
\begin{defi} \label{defnest}
Suppose that  $\n:=n_{1},n_{2},\dots,n_{r}$ is a sequence of $r\ge 1$ nonnegative integers, and $\betab:=\beta_1,\dots,\beta_{r-1}$ is a sequence of classes in $H^2(S,\ZZ)$.  The nested Hilbert scheme is the closed subscheme
\begin{equation}\label{nested inclusion} \iota:\S{\n}_{\betab} \hookrightarrow \S{n_1}_{\beta_1}\times\cdots \times  \S{n_{r-1}}_{\beta_{r-1}}\times  \S{n_r} \end{equation} naturally defined by the $r$-tuples $(Z'_1,\dots, Z'_r)$ of subschemes of $S$ such that  $\pts(Z'_i) \subset  Z'_{i-1}$ is a subscheme for all $1<i\le r$. We drop $\betab$ or $\n$ from the notation respectively when $\beta_i=0$ for all $i$ or $n_i= 0, \beta_j\neq 0$ for all $i,j$. The scheme $\S{\n}_\betab$ represents the functor that takes a scheme
$U$ to the set of flat families of ideals $\I_1,\dots, \I_r\subseteq \O_{S\times U}$ and flat families of Cartier divisors $\D_1\dots,\D_{r-1} \subset S\times U$ such that  $$\I_1(-\D_1-\D_2-\cdots-\D_{r-1}) \subseteq  \I_2(-\D_2-\cdots-\D_{r-1})\subseteq \cdots \subseteq \I_{r-1}(-\D_{r-1}) \subseteq \I_r,$$ and on restriction to any closed fiber $\I_i|_{S \times \{u\}}$ has colength $n_i$ and $[\D_i|_{S \times \{u\}}]=\beta_i$.
\end{defi}
As a set we can think of $\S{\n}_{\betab}$ as given by the tuples of subschemes $$(Z_1,\dots, Z_r)\in \S{n_1}\times\cdots \times \S{n_r},\quad (C_1,\dots,C_{r-1})\in S_{\beta_1} \times\cdots \times S_{\beta_{r-1}}$$ together with the nonzero maps $\phi_{i}: I_{Z_{i}}\to I_{Z_{i+1}}(C_{i})$,  up to multiplication by scalars, for all  $1\le i< r.$ Note that each $\phi_i$ is necessarily injective, and taking double dual, it gives (up to a scalar) the tautological section of $\O(C_i)$. In this correspondence, $Z'_r=Z_r$, and for $1\le i <r$ the ideal sheaf of the subscheme $Z'_i$ is $I_{Z_i}(-C_i)$.  


By the construction of nested Hilbert schemes, the maps $\phi_i$ above are induced from the universal maps
$$\Phi_i: \i{n_i}\to \i{n_{i+1}}_{\beta_i} \quad  \quad 1\le i< r$$ defined over $S\times \S{\n}_{\betab}$. 
Applying the functors $\dR\hom_\pi\big (-,\i{n_{i+1}}_{\beta_i}\big)$ and $\dR\hom_\pi\big (\i{n_{i}},-\big)$ to the universal map $\Phi_i$, we get the following morphisms in derived category $$\dR\hom_\pi\big(\i{n_{i+1}}, \i{n_{i+1}}\big ) \xrightarrow{\Xi_i} \dR \hom_\pi\big(\i{n_i}, \i{n_{i+1}}_{\beta_i}\big),$$ $$\dR\hom_\pi\big(\i{n_{i}}, \i{n_{i}}\big ) \xrightarrow{\Xi'_i} \dR \hom_\pi\big(\i{n_i}, \i{n_{i+1}}_{\beta_i}\big).$$ 
Consider the map \begin{align}\label{matr}\bigoplus_{i=1}^r \dR\hom_\pi \big(\i{n_i},\i{n_i}\big) \xrightarrow{\tiny \left[\begin{array}{cccccc}-\Xi'_1 & \Xi_1 & 0& \dots &0 &0\\
0 & -\Xi'_2 & \Xi_2 & 0 & \dots &0\\ 
\cdot & \cdot & \dots & \cdot & \cdot & \cdot \\
\cdot & \cdot & \dots & \cdot & \cdot & \cdot \\
\cdot & \cdot & \dots & \cdot & \cdot & \cdot \\
0 & 0& \dots & 0 & -\Xi'_{r-1} & \Xi_{r-1} \\
 \end{array}\right]} \bigoplus_{i=1}^{r-1}\dR \hom_\pi \big(\i{n_i}, \i{n_{i+1}}_{\beta_i}\big).\end{align}
We will show that this map factors through the trace free part \begin{align*}\left[\bigoplus_{i=1}^r \dR\hom_\pi \left(\i{n_i},\i{n_i}\right)\right]_0:&=\cone\left(\bigoplus_{i=1}^r\dR\hom_\pi \big(\i{n_i},\i{n_i}\big)\xrightarrow{[\tr \;\dots\;\tr]} \dR\pi_*\O \right)[-1]\\ &\cong\cone\left(\dR\pi_*\O \xrightarrow{[\id \; \dots\; \id]^t} \bigoplus_{i=1}^r\dR\hom_\pi \big(\i{n_i},\i{n_i}\big) \right)
.\end{align*}

The following proposition that implies Theorem \ref{thm1} is proven in Section \ref{r-step}:
\begin{prop} \label{abs-r}
$\S{\n}_\betab$ is equipped with a perfect absolute obstruction theory $\Fb\to \LL_{\S{\n}_\betab}$ whose virtual tangent bundle is given by\begin{align*}&\Fbv\cong \cone \left(\left[\bigoplus_{i=1}^r \dR\hom_\pi \big(\i{n_i},\i{n_i}\big)\right]_0 \to \bigoplus_{i=1}^{r-1}\dR \hom_\pi \big(\i{n_i}, \i{n_{i+1}}_{\beta_i}\big)\right),\end{align*}
where the map is the one defined above.\footnote{The negative signs on the diagonal of the matrix in \eqref{matr} were missing in the first draft of the paper. We were notified of the corrected form of the matrix by Richard Thomas.}
\end{prop}
\subsection{$2$-step nested Hilbert schemes}  \label{2-step} In this section we study $\S{n_1,n_2}_{\beta}:=\S{\n}_{\betab}$ in the case $r=2$. Recall from Definition \ref{defnest} that for a pair of nonnegative integers $n_1, n_2$ and an effective curve class $\beta \in H^2(S,\ZZ)$, we defined the projective scheme $\S{n_1,n_2}_\beta$, whose set of closed points is given by $$\left\{(Z_1, C, Z_2,\phi) \mid Z_i \in \S{n_i}, \;C\in S_\beta, \;\;  \phi\in \PP(\Hom(I_{Z_1},I_{Z_2}(C))) \right\}.$$  There are universal objects defined over $S\times \S{n_1,n_2}_\beta$ as before: 
$$ \Phi: \i{n_1}\to \i{n_{2}}_\beta.$$ 
  
 Applying the functors $\dR\hom_\pi(-,\i{n_2}_\beta)$ and $\dR\hom_\pi (\i{n_1}, -)$ to the universal map $\Phi$, we get the following morphisms in derived category \begin{align} \label{morph1}&\dR\hom_\pi\big(\i{n_2}, \i{n_2}\big ) \xrightarrow{\Xi}\dR \hom_\pi\big(\i{n_1}, \i{n_2}_\beta\big) \\\notag &\dR\hom_\pi\big(\i{n_1}, \i{n_1}\big) \xrightarrow{\Xi'} \dR \hom_\pi\big(\i{n_1}, \i{n_2}_\beta \big). \end{align}

Let $T$ be any scheme over $\C$-scheme $U$, and let 
\begin{equation*}
\xymatrix@=1em{
T \ar[d]_-a \ar@{^{(}->}[r] & \Tb \ar[ld]^-{\overline{a}}\\
U} 
\end{equation*}  be a square zero extension over $U$ with the ideal $J$. As $J^2=0$, $J$ can be considered as an $\O_{T}$-module. Suppose we have a Cartesian diagram

\begin{equation}  \label{sqzero}
\xymatrix{
T \ar[d]^a \ar[r]^-g & \S{n_1,n_2}_\beta \ar[d]^{\p:=\pts\circ \pr_1}\\
U \ar[r] & \S{n_1},} 
\end{equation} where $\pr_1$ is the composition of $\iota$ with the projection $\S{n_1}_\beta \times \S{n_2}\to \S{n_1}_\beta$. The bottom row of \eqref{sqzero} corresponds to a flat $U$-family $Z_{1,U}\subset S\times U$ of subschemes of length $n_1$, and the top row corresponds to the data  
\begin{equation} \label{Tpoint} (Z_{1,T}, C_T, Z_{2,T},\phi_T),\end{equation}  consisting of $Z_{1,T}=Z_{1,U}\times_{S\times U,(\id,a)} S\times T,$ a $T$-flat family $Z_{2,T}\subset S\times T$  of subschemes of length $n_2$,  a $T$-flat family $C_T\subset S\times T$  of effective Cartier divisors in class $\beta$, and  (up to a scalar)  a homomorphism$$\phi_T\colon I_{Z_{1,T}}\to I_{Z_{2,T}}(C_T).$$ Let $\pi_T$ be the projections from $S\times T\to  T$. By \cite[Prop. IV.3.2.12]{Ill} and \cite[Thm 12.8]{JS12}, there exists an element $$\ob:=\ob( \phi_T,J) \in \Ext^2_{S\times T}\left(\coker(\phi_T), \pi^*_TJ \otimes I_{ Z_{2,T}}(C_T)\right),$$ whose vanishing is necessary and sufficient to extend the $T$-family \eqref{Tpoint} to a  $\Tb$-family \begin{equation} \label{Tbpoint} (Z_{1,\Tb}, C_{\Tb}, Z_{2,\Tb},\phi_{\Tb}) \end{equation} such that $Z_{1,\Tb}=Z_{1,U}\times_{S\times U,(\id,\overline{a})} S\times \Tb$.  In fact by \cite[Prop. IV.3.2.12]{Ill}, $\ob$ is the obstruction to deforming the morphism $\phi_T$ while the deformation $I_{Z_{1,\Tb}}$ of $I_{Z_{1,T}}$  is given. Suppose that $\phi_{\Tb}: I_{Z_{1,\Tb}} \to \F$ is such a deformation, where $ \F$ is a flat family of rank 1 torsion free sheaves with $\F|_{S\times T}=I_{Z_{2,T}}(C_T)$. Then by \cite[Lemma 6.13]{K90} the double dual $\F^{**}$ is an invertible sheaf. Now $\phi_{\Tb}^{**}:\O_{S\times \Tb}\to \F^{**}$ is  injective when restricted to closed fibers of $S\times \Tb\to \Tb$, and hence by \cite[Lemma 2.1.4]{HL10}, $\coker(\phi_{\Tb}^{**})$  is also flat over $\Tb$. Thus, there exists a $\Tb$-flat subscheme $C_{\Tb}\subset S\times \Tb$ (i.e. the Cartier divisor cut out by $\phi_{\Tb}^{**}$) that restricts to $C_{T}$ and $\F^{**}\cong \O(C_{\Tb})$. We conclude from $\F\subseteq \F^{**}$ that $\F\cong I_{Z_{2,\Tb}}(C_{\Tb})$ for some $\Tb$-flat subscheme $Z_{2,\Tb}\subset S\times \Tb$ restricting to $Z_{2,T}$. 

If $\ob = 0$ then by \cite[Prop. IV.3.2.12]{Ill} again, the set of isomorphism classes of deformations forms a torsor under $$\Ext^1_{S\times T}\left(\coker(\phi_T), \pi_T^*J \otimes I_{ Z_{2,T}}(C_T)\right).$$ 
\begin{lem} \label{atred}
$\ob=\at_{\red}(\phi_T)\cup(\pi^*_Te(\Tb)\otimes \id)$, where
  \begin{equation}\label{atiyah}\at_{\red}( \phi_T)\in \Ext^1_{S\times T}\left(\coker(\phi_T),\pi^*_T\LL_{a}\otimes I_{ Z_{2,T}}(C_T)\right),\end{equation}  is a certain reduction of the Atiyah class  $\at_{\O_{S\times T}/\O_{S\times U}}(\phi_T)$ of $\phi_T$ \cite[IV.2.3]{Ill}, and $$\pi^*_Te(\Tb)\otimes \id \in \Ext^1_{S\times T}(\pi^*_T\LL_{a}\otimes I_{ Z_{2,T}}(C_T),\pi^*_T J\otimes I_{ Z_{2,T}}(C_T)),$$ in which 
\begin{equation}\label{cotsqzero} e(\Tb)\in \Ext^1_T(\LL_{a},J)\end{equation} is  the Kodaira-Spence class associated to the square zero extension $T\hookrightarrow \Tb$. 
\end{lem}
\begin{proof}
The proof is the same as the proof of \cite[Thm 12.9]{JS12}: in diagram (12.17) of [ibid], the first vertical arrow needs to  be replaced by $$\at_{\O_{S\times T}/\O_{S\times U}}(\phi_T):\coker(\phi_T)\to k^1\left(\mathbb{L}^{\bullet,\text{gr}}_{(\O_{S\times T}\oplus I_{Z_{1,T}})/\O_{S\times U}}\otimes (\O_{S\times T}\oplus I_{Z_{2,T}}(C_T))\right)[1].$$ This is the degree 1 part of the transitivity triangle (see item 1 in Subsection \ref{notcon}) associated to natural homomorphism of sheaves of graded algebras on $S\times T$ $$\O_{S\times U} \to \O_{S\times T}\oplus  I_{Z_{1,T}}\xrightarrow{[\id\; \phi_T]}\O_{S\times T}\oplus I_{Z_{2,T}}(C_T).$$   
By flatness of $I_{Z_{1,U}}$ over $\O_U$, we see that $\O_{S\times U}\oplus I_{Z_{1,U}}$ is flat over $\O_U$ and hence $$(\O_{S\times U}\oplus I_{Z_{1,U}})\overset{\dL}{\otimes}_{\O_{S\times U}} \O_{S\times T}\cong(\O_{S\times U}\oplus I_{Z_{1,U}})\otimes_{\O_{S\times U}} \O_{S\times T}\cong \O_{S\times T}\oplus I_{Z_{1,T}},$$ so \cite[II.2.2]{Ill} $\mathbb{L}^{\bullet,\text{gr}}_{(\O_{S\times T}\oplus I_{Z_{1,T}})/\O_{S\times U}}$ is isomorphic to $$
\Big( \mathbb{L}^{\bullet}_{\O_{S\times T}/\O_{S\times U}}\otimes(\O_{S\times T}\oplus I_{Z_{1,T}})\Big)\oplus \Big(\mathbb{L}^{\bullet,\text{gr}}_{(\O_{S\times U}\oplus I_{Z_{1,U}})/\O_{S\times U}}\otimes(\O_{S\times T}\oplus I_{Z_{1,T}})\Big),$$ and hence as in the proof of \cite[Thm 12.9]{JS12}, composing $\at_{\O_{S\times T}/\O_{S\times U}}(\phi_T)$ with the projection $$\mathbb{L}^{\bullet,\text{gr}}_{(\O_{S\times T}\oplus I_{Z_{1,T}})/\O_{S\times U}}\to\mathbb{L}^{\bullet}_{\O_{S\times T}/\O_{S\times U}}\cong \pi^*_T\LL_{a},$$ we arrive at the definition of $\at_{\red}(\phi_T)$ in \eqref{atiyah}.  

\end{proof}

Now the idea is that from the deformation/obstruction theory of the universal map $\Phi:\i{n_1}\to \i{n_2}_{\beta}$ reviewed above, we construct  a \emph{relative} perfect obstruction theory $\Fbr \to \LL_{\S{n_1,n_2}_\beta/\S{n_1}}$ (Proposition \ref{rel}). Using the fact that $\S{n_1}$ is nonsingular, we will deduce the \emph{absolute} perfect obstruction theory $\Fb \to \LL_{\S{n_1,n_2}_\beta}$ from $\Fbr$ (Proposition \ref{abs}). This will prove Proposition \ref{abs-r} for $r=2$.

\begin{prop} \label{rel}
The complex $$\Fbr:=\cone(\Xi)^\vee\cong \dR\hom_\pi\left(\coker(\Phi),\i{n_2}_\beta\right)^\vee[-1]$$ defines a relative perfect obstruction theory for the morphism $\p:\S{n_1,n_2}_\beta \to \S{n_1}$. In other words, $\Fbr$ is perfect with amplitude contained in $[-1,0]$, and there exists a morphism in the derived category  $\alpha: \Fbr\to \LL_{\p}$, such that $h^0(\alpha)$ and $h^{-1}(\alpha)$ are respectively isomorphism and epimorphism. The rank of $\Fbr$ is equal to $$\rank\left[\Fbr\right]=n_2-n_1-\frac{\beta \cdot \beta^D}{2}.$$ \end{prop}
\begin{proof}
\textbf{Step 1:} \emph{(perfectness)} We show that the complex $\Fbrv$ is perfect with amplitude $[0,1]$. By basechange and the same argument as in the proof of \cite[Lemma 4.2]{HT10}, it suffices to show that $h^i(\dL t^* \Fbrv)=0$ for $i\neq 0, 1$, where $t:P\hookrightarrow \S{n_1,n_2}_\beta$ is the inclusion of an arbitrary closed point $P=(Z_1,C,Z_2,\phi)\in \S{n_1,n_2}_\beta$. Therefore, by the definition of $\Fbrv$ we get the exact sequence 
$$\dots \to \Ext^{i}_S(I_{Z_1},I_{Z_2} (C)) \to h^i(\dL t^* \Fbrv) \to \Ext^{i+1}_S(I_{Z_2},I_{Z_2})\to \dots .$$
All the $\Ext^i_S$ for $i\neq 0, 1, 2$ vanish, so we deduce easily that $h^i(\dL t^* \Fbrv)=0$ for $i\neq -1, 0, 1,  2$. From the sequence above we see that $$h^{-1}(\dL t^* \Fbrv)=\ker\big(\Hom_S(I_{Z_2},I_{Z_2}) \to \Hom_S(I_{Z_1},I_{Z_2}(C))\big).$$ But by definition this map is induced by applying $\Hom_S(-,I_{Z_2}(C))$ to the map  $\phi:I_{Z_1}\to I_{Z_2}(C)$, so it takes $\id$ to $\phi$, and hence it is injective. Therefore, $h^{-1}(\dL t^* \Fbrv)=0$. 

To prove $h^2(\dL t^* \Fbrv)=0$, we show that the map $$\Ext^{2}_S(I_{Z_2},I_{Z_2})\to \Ext^{2}_S(I_{Z_1},I_{Z_2} (C))$$ in the exact sequence above is surjective, or equivalently by Serre duality, the dual map $$\Hom_S(I_{Z_2},I_{Z_1}\otimes \omega_S(-C))\to \Hom_S(I_{Z_2},I_{Z_2}\otimes \omega_S)$$ is injective.  But this follows after applying the left exact functor $\Hom_S(I_{Z_2},-)$ to the injection $I_{Z_1}\otimes \omega_S(-C)\to I_{Z_2}\otimes \omega_S$ that is induced  by tensoring the map $\phi$ above by $ \omega_S(-C)$. 

\textbf{Step 2}: \emph{(map to the cotangent complex)} We construct a morphism in derived category $\alpha: \Fbr\to \LL_{\p}$. Consider the reduced Atiyah class \eqref{atiyah} in the case $T=\S{n_1,n_2}_\beta$ and $U=\S{n_1}$. It defines an element in 

\begin{align*}
&\Ext^{1}_{S\times \S{\n}_\beta}\big(\coker(\Phi),  \pi^{*}\LL_{\p}\otimes \i{n_2}_\beta\big)\cong\\
&
 \Ext^{1}_{S\times \S{\n}_\beta}\big(\dR \hom\big(\i{n_2}_\beta, \coker(\Phi)\big), \pi^{*}\LL_{\p}\big)\cong \quad(\text{by the definition of $\pi^!$})\\
&
\Ext^{1}_{S\times \S{\n}_\beta}\big(\dR \hom \big(\i{n_2}_\beta, \coker(\Phi)\otimes \omega_{\pi}[2]\big), \pi^{!}\LL_{\p}\big)\cong \quad(\text{by Grothendieck-Verdier duality})\\
&
 \Ext^{1}_{\S{\n}_\beta}\big(\dR \hom_\pi \big(\i{n_2}_\beta , \coker(\Phi)\otimes \omega_{\pi}[2]\big), \LL_{\p}\big)\cong\\
&
\Hom_{\S{\n}_\beta} \big(\dR \hom_\pi \big(\i{n_2}_\beta, \coker(\Phi)\otimes \omega_{\pi}[1]\big),  \LL_{\p}\big).
\end{align*}
So under the identification above, the reduced Atiyah class defines a morphism in derived category
$$\alpha: \dR \hom_\pi \big(\i{n_2}_\beta, \coker(\Phi)\otimes \omega_{\pi}[1]\big)\to  \LL_{\p}.$$ But by Grothendieck-Verdier duality again,
$$\dR \hom_\pi \big(\i{n_2}_\beta, \coker(\Phi)\otimes \omega_{\pi}[1]\big)\cong \dR \hom_\pi \big(\coker(\Phi),\i{n_2}_\beta\big)^\vee[-1] \cong\Fbr,$$ and hence we are done.


\textbf{Step 3:} \emph{(obstruction theory)} We show $h^0(\alpha)$ and $h^{-1}(\alpha)$ are respectively isomorphism and epimorphism. We use the criterion in  \cite[Theorem 4.5]{BF97}. 
Suppose we are in the situation of the diagram \eqref{sqzero}. Define $$f:=(\id,g):S\times T\to S\times \S{n_1,n_2}_\beta.$$ 
Composing $e(\Tb)$ (given in \eqref{cotsqzero}) and the natural morphism of cotangent complexes $\dL g^*\LL_{\p} \to \LL_a$ gives the element $\varpi(g) \in \Ext^1_T(\dL g^*\LL_{\p},J)$ whose image under $\alpha$ is denoted by $$\alpha^*\varpi(g)\in \Ext^{1}_T\left(\dL g^*\Fbr, J\right).$$

For $i=0, 1$, we will use the following identifications:
 \begin{align*}
 \Ext^{i}_T\big(\dL g^*\Fbr, J\big)&\cong \Ext^{i}_T\left(\dL g^*\dR \hom_\pi\big(\i{n_2}_\beta, \coker(\Phi)\otimes \omega_{\pi}[1]\right), J\big)\\
 &
\cong \Ext^{i}_{\S{\n}_\beta}\big( \dR \hom_\pi\big(\i{n_2}_\beta,  \coker(\Phi)\otimes \omega_{\pi}[1]\big), \dR g_*J\big)\\
&
\cong \Ext^{i}_{S\times \S{\n}_\beta}\big(\dR \hom \big(\i{n_2}_\beta, \coker(\Phi)\otimes \omega_{\pi}[1]\big), \pi^!\dR g_*J\big)\\
&
\cong \Ext^{i+1}_{S\times \S{\n}_\beta}\big(\dR \hom \big(\i{n_2}_\beta, \coker(\Phi)\big), \pi^* \dR g_*J\big).
\end{align*}
Here we have used the fact that $\dL g^* \dashv \dR g_*$ i.e. $\dL g^*$ is the left adjoint of $\dR g_*$, and Grothendieck-Verdier duality. Now  using $ \dR f_{*}\pi_T^*=\pi^* \dR g_*$  in the last $\Ext$ above, we get  

\begin{align*}
& \Ext^{i+1}_{S\times \S{\n}_\beta}\big(\dR \hom \big(\i{n_2}_\beta, \coker(\Phi)\big), \dR f_{*}\pi_T^*J\big) \cong \quad \text{(by $\dL f^* \dashv \dR f_*$)}\\
& \Ext^{i+1}_{S\times T}\big(\dL f^* \dR \hom\big(\i{n_2}_\beta, \coker(\Phi)\big), \pi_T^*J\big)\cong \\
& \Ext^{i+1}_{S\times T}\big( \dR\hom\big(\dL f^*\i{n_2}_\beta, \dL f^*\coker(\Phi)\big), \pi_T^*J\big) \cong\\
& \Ext^{i+1}_{S\times T}\big(\dL f^*\coker(\Phi), \pi_T^*J \otimes \dL f^*\i{n_2}_\beta\big)\cong \quad (\text{by flatness of $\coker(\Phi)$ and $\i{n_2}_\beta$})\\
& \Ext^{i+1}_{S\times T}\big( \coker(\phi_T), \pi_T^*J \otimes  I_{Z_{2,T}}(C_T)\big).
 \end{align*}
 Similar to the Step 2 it can be seen that the composition $$\dL g^*\Fbr \xrightarrow{g^*\alpha}\dL g^*\LL_{\p} \to \LL_a$$ arises from $\at_{\red}(\phi_T)$ over $S\times T$ (see \eqref{atiyah}).
Therefore, $$\ob(\phi_T,J)=\at_{\red}(\phi_T) \cup (\pi^*_Te(\Tb)\otimes \id),$$ 
is identified with the element $\alpha^{*}\varpi(g)$ via the identifications above for  $i=1$. By the definition of the obstruction class $\ob(\phi_T,J)$, this means that $\alpha^{*}\varpi(g)$ vanishes if and only if there exists an extension $\overline{g}$ of $g$ corresponding to \eqref{Tbpoint}. Using the identifications above, this time for $i=0$, we can see that if $\alpha^{*}\varpi(g)=0$, then the set of extensions forms a torsor under $\Hom_T(\dL g^*\Fb, J)$. Now by \cite[Theorem 4.5]{BF97}  $\alpha$ is an obstruction theory.

\textbf{Step 4:} \emph{(rank of $\Fbr$)} The claim about the rank follows from
\begin{align*}
\rank \left[\Fbr\right]&=\rank \left[\cone(\Xi)\right] \\&=\rank \left[\dR \hom_\pi\big(\i{n_1}, \i{n_2}_\beta\big)\right]-\rank\left[ \dR\hom_\pi\big(\i{n_2}, \i{n_2}\big )\right] 
\\&=\chi(I_{Z_1},I_{Z_2}(C))- \chi(I_{Z_2},I_{Z_2})\\&= -n_1-n_2+\chi(\O_S(C))+2n_2-\chi(\O_S)
\\&=n_2-n_1-\beta\cdot K_C/2+\beta^2/2,
\end{align*} where $(Z_1,C,Z_2,\phi)$ is a closed point of $\S{n_1,n_2}_\beta$. 

\end{proof}


\begin{prop} \label{abs}
$\S{n_1,n_2}_\beta$ is equipped with the absolute perfect obstruction theory $\Fb \to \LL_{\S{n_1,n_2}_\beta}$. Its virtual tangent bundle is given by 
\begin{align*}\Fbv\cong\cone \bigg( \Big[\dR\hom_\pi \big(\i{n_1},\i{n_1}\big) &  \oplus \dR\hom_\pi\big(\i{n_2},\i{n_2}\big)\Big]_0 \\ &\xrightarrow{[-\Xi' \; \Xi]} \dR \hom_\pi \big(\i{n_1}, \i{n_2}_\beta\big)\bigg).\end{align*}
\end{prop}
\begin{proof}  
Since $\S{n_1}$ is nonsingular, by the standard techniques (see \cite[Section 3.5]{MPT10})
\begin{equation} \label{reltoabs}\Fb:=\cone\left( \Fbr\xrightarrow{\theta}  \p^* \Omega_{\S{n_1}}[1]\right)[-1]\end{equation} gives  a perfect absolute obstruction theory over $\S{n_1,n_2}_\beta$, where $\theta$ is the composition of $\alpha:\Fbr \to \LL_{\p}$ and the Kodaira-Spencer map $c:\LL_{\p}\to \p^*\Omega_{\S{n_1}}[1]$.  We claim that $\theta$ is  given by
\begin{align*} \Fbr\cong &\dR \hom_\pi \big(\i{n_2}_\beta, \coker(\Phi)\otimes \omega_{\pi}\big)[1]\to\dR \hom_\pi \big(\i{n_1}, \coker(\Phi)\otimes \omega_{\pi}\big)[1]\\&\to \dR\hom_{\pi}\big(\i{n_1},\i{n_1}\otimes \omega_\pi\big)_0[2]\cong \p^*\Omega_{\S{n_1}}[1],\end{align*} where the first and second maps are respectively induced by $\Phi:\i{n_1}\to \i{n_2}_\beta$ and the natural map $c':\coker(\Phi)\to \i{n_1}[1]$. 
To see the claim consider the commutative diagram of graded sheaves of algebras on $S\times \S{n_1,n_2}_\beta$ with all unlabelled arrows are the obvious natural maps\footnote{following our convenction we have suppressed the symbols for pullbacks via natural morphisms.}:
\begin{equation}  \label{ddd}
 \xymatrix@=2em{\O_{ S\times \S{n_1}} \ar[r]& \O_{S\times \S{n_1,n_2}_{\beta}}\oplus \i{n_1} \ar[r]^-{(\id,\Phi)}& \O_{S\times \S{n_1,n_2}_\beta} \oplus \i{n_2}_{\beta}\\
 \O_{S } \ar[r] \ar[u]& \O_{S\times \S{n_1,n_2}_\beta}  \ar[r] \ar[u]& \O_{S\times \S{n_1,n_2}_\beta}\oplus \i{n_2}_\beta,\ar@{=}[u]\\
\O_{S } \ar[r] \ar@{=}[u]& \O_{S\times \S{n_1,n_2}_\beta}  \ar[r] \ar@{=}[u]& \O_{S\times \S{n_1,n_2}_\beta}\oplus \i{n_1} \ar[u]_-{(\id,\Phi)}
}\end{equation}
Taking the transitivity triangles (see item 1 in Subsection \ref{notcon}) of the top two rows and applying the same construction leading to the definition of $\at_{\red}(\Phi)$ in the proof of Lemma \ref{atred}, we obtain the commutative diagram of sheaves on $S\times \S{n_1,n_2}_\beta$
$$\xymatrix@C=65pt@R=25pt{
\coker(\Phi)  \ar[r]^-{\at_{\red}(\Phi)} &  \LL_{\p}\otimes \i{n_2}_\beta [1]\\
\i{n_2}_\beta   \ar[u] \ar[r]  & \LL_{\S{n_1,n_2}_\beta}\otimes \i{n_2}_\beta [1] \ar[u]}
$$ 
in which both vertical arrows are the natural maps, the bottom horizontal arrow  is the pullback of the Atiyah class  $$\at(\i{n_2}_\beta) \in \Ext^1(\i{n_2}_\beta, \LL_{\S{n_2}_\beta} \otimes\i{n_2}_\beta)$$ composed with $  \q^* \LL_{\S{n_2}_\beta}\to \LL_{\S{n_1,n_2}_\beta}$ induced by the natural morphism $\q\colon \S{n_1,n_2}_\beta\to \S{n_2}_\beta$. Taking cones over the vertical arrows in turn induces the commutative diagram
\begin{equation}  \label{ddd1}\xymatrix@C=65pt@R=25pt{
 \i{n_1}[1] \ar[r] &  \p^*\Omega_{\S{n_1}}\otimes \i{n_2}_\beta [2]\\
\coker(\Phi)   \ar[u]^{c'} \ar[r]^-{\at_{\red}(\Phi)}  & \LL_{\p}\otimes \i{n_2}_\beta [1]. \ar[u]_-{c\otimes \id}}
\end{equation}  
On the other hand, the transitivity triangles of the bottom two rows of \eqref{ddd} imply that the top row in \eqref{ddd1} is in fact the composition  $$\i{n_1}\xrightarrow{\at(\i{n_1})}\p^*\Omega_{\S{n_1}}\otimes \i{n_1}[1]\xrightarrow{\id \otimes \Phi}\p^*\Omega_{\S{n_1}}\otimes \i{n_2}_\beta[1],$$ where $\at(\i{n_1})$ is the (pullback of the) usual Atiyah class. The claim now follows from this and the construction of the map $\alpha$ using the class $\at_{\red}(\Phi)$ in 
Step 2 of proof of Proposition \ref{rel}. To ease the notation let $$\Ab:=\dR\hom_\pi \big(\i{n_1},\i{n_1}\big), \Bb:=\dR\hom_\pi\big(\i{n_2},\i{n_2}\big), \Cb:=\dR \hom_\pi \big(\i{n_1}, \i{n_2}_\beta\big),$$
and denote by $\Eb$ the right hand side of the expression in the proposition. By Proposition \ref{rel}, $\Fbrv=\cone\left(\Bb\to \Cb\right)$, so by \eqref{tanhilb} and the claim above, \eqref{reltoabs}  can be rewritten as 
\begin{align*}
\Fbv=&\cone \left(\Ab_0 \xrightarrow{\theta^\vee} \cone\left(\Bb\to \Cb\right)\right).
\end{align*}


%
%

Consider the commutative diagram 
$$\xymatrix{
& \dR\pi_* \O  \ar[d]^{[\id\; \id]^t}  \ar@{=}[r] & \dR\pi_* \O \ar[d]^{\id}\\
\Bb   \ar[r] &  \Ab\oplus \Bb \ar[r] & \Ab}.$$
in which the bottom row is the natural exact triangle. Taking the cone of the diagram one gets the exact triangle \begin{align} \label{trfrtr}\Bb \to [\Ab\oplus \Bb]_0\to \Ab_0.\end{align}
Next consider the commutative diagram
$$\xymatrix@=2em{
\Bb  \ar@{=}[d] \ar[r]^-{[\id \; 0]^t} &  \Ab\oplus \Bb \ar[d]^{[-\Xi' \;\Xi]} \ar[r] & \Ab_0\oplus \dR\pi_* \O \ar[d]^{[\theta^\vee \; 0]}\\
\Bb \ar[r]^{\Xi} & \Cb \ar[r] &\cone \left(\Bb\to \Cb\right)  }
$$
 in which both rows are exact triangles, and in the rightmost vertical arrow we use the splitting $\Ab\cong \Ab_0\oplus \dR\pi_* \O$ given by the trace map  (the left square is obviously commutative, and the right square is commutative by the claim we proved above). Now by construction the vertical arrows in the above diagram factor through the exact triangle \eqref{trfrtr}, and hence we arrive at the following commutative diagram in which all the rows and columns are exact triangle.

$$\xymatrix@=1em{
\Bb  \ar@{=}[d] \ar[r] &  [\Ab\oplus \Bb]_0 \ar[d]^{[-\Xi' \;\Xi]} \ar[r] & \Ab_0 \ar[d]^{\theta^\vee}\\
\Bb \ar[r]^{\Xi} & \Cb \ar[d] \ar[r] &\cone \left(\Bb\to \Cb\right) \ar[d]  \\
& \Eb \ar[r]  & \Fbv.}
$$
In fact the columns and the top and middle rows are exact triangles with commutative top squares, and the bottom row is induced from the rest of the diagram by taking the cone. Therefore, the bottom row must also be an exact triangle which means that $\Eb\cong \Fbv$ as desired.
\end{proof}
This finishes the proof of Proposition \ref{abs-r} in the case $r=2$.
Propositions \ref{rel} and \ref{abs} imply 
\begin{cor} \label{virclass}
The perfect obstruction theory of Proposition \ref{abs} defines a virtual fundamental class on $\S{n_1,n_2}_\beta$ denoted by $$[\S{n_1,n_2}_\beta]^{\vir}\in A_d(\S{n_1,n_2}_\beta), \quad  \quad d=n_1+n_2-\frac{\beta \cdot \beta^D}{2}.$$

\end{cor} \qed

\subsection{Gillam's construction} \label{gillam} The relative obstruction theory $\Fbr$ is obtained from the deformation/obstruction theory of the universal map $\Phi:\i{n_1}\to \i{n_2}_{\beta}$. 
As mentioned in the introduction, following \cite{G11}, one can can instead use deformation/obstruction theory of the quotient $\coker(\Phi)$  to construct a relative perfect obstruction theory  $$F'^\bullet_{\rel} \to \LL_{\S{n_1,n_2}_\beta/\S{n_2}}.$$  This amounts to identifying $\S{n_1,n_2}_\beta$ with a component of the relative quot scheme $$\operatorname{Quot}_{S\times \S{n_2}/\S{n_2}}(\i{n_2})$$ of quotients of $\i{n_2}$ and applying \cite[Theorem 4.6]{G11}. By a similar argument as Proposition \ref{abs} (using the smoothness of $\S{n_2}$ this time), one can then deduce an absolute perfect obstruction theory $F'^\bullet\to \LL_ {\S{n_1,n_2}_\beta}$. By comparing the K-theory classes of $F^\bullet$ and $F'^\bullet$ the reader can verify that the resulting virtual class from this approach is the same as that of Corollary \ref{virclass} (see \cite{S04}).

\subsection{Reduced perfect obstruction theory} \label{sec:reduced}
In this section we assume that for any effective line bundle $L$ on $S$ with $c_1(L)=\beta$, we have \begin{equation}\label{condition} |L^D|=\emptyset \quad \text{or equivalently} \quad H^2(L)=0. \end{equation}
Recall that the relative virtual tangent bundle of Proposition \ref{rel} is given by 
$$\Fbrv=\cone\bigg(\dR\hom_\pi\big(\i{n_2}_\beta, \i{n_2}_\beta\big ) \to \dR \hom_\pi\big(\i{n_1}, \i{n_2}_\beta\big)\bigg).$$  We get a natural map $$\mu: \Fbrv\to \dR\hom_\pi\big(\i{n_2}_\beta, \i{n_2}_\beta\big )[1],$$ that induces \begin{align*}h^1(\mu): h^1(\Fbrv)\cong & \;\ext^2_{\pi}\big(\i{n_2}_\beta/\i{n_1},\i{n_2}_\beta\big)\to \\&\ext^2_\pi\big (\i{n_2}_\beta,\i{n_2}_\beta\big) \cong \dR^2\pi_*\O_{S\times \S{n_1,n_2}_\beta}\cong \O_{\S{n_1,n_2}_\beta}^{ p_g}.\end{align*} We claim that $h^1(\mu)$ is surjective. To see this, by basechange, it suffices to prove that $h^1(\mu)$ is fiberwise surjective. Let  $t:P\hookrightarrow \S{n_1,n_2}_\beta$ be the inclusion of an arbitrary closed point $P=(Z_1,C,Z_2,\phi)\in \S{n_1,n_2}_\beta$. Then, by basechange we have the natural exact sequence\footnote{Note that $\Ext^3_S(\coker (\phi),I_{Z_2}(C))=\Ext^3_S(I_{Z_2}(C),I_{Z_2}(C))=0$.}
$$\dots \to h^1(\dL t^*\Fbrv)\xrightarrow{h^1(\mu)_P} \Ext^2_S(I_{Z_2}(C),I_{Z_2}(C))\xrightarrow{u} \Ext^2_S(I_{Z_1},I_{Z_2}(C))\to 0.$$
The surjectivity of the map $u$ was established in Step 1 of the proof of Proposition \ref{rel}. We have 
 $$\Ext^2_S(I_{Z_2}(C),I_{Z_2}(C))\cong \Ext^2_S(I_{Z_2},I_{Z_2})\cong H^2(\O_S),$$ $$\Ext^2_S(I_{Z_1},I_{Z_2}(C))^*\cong\Hom_S(I_{Z_2},I_{Z_1}(C)^D) \subseteq \Hom_S(I_{Z_2},\O_S(C)^D)\cong H^0(\O_S(C)^D).$$ By assumption \eqref{condition}, $H^0(\O_S(C)^D)=0$, and hence $h^1(\mu)_P$ is surjective and the claim follows. 
We now have the diagram
 $$ \xymatrix{\Ext^1_S(I_{Z_1},I_{Z_1})_0 \ar[r]^-{h^1(\theta^\vee)_P}& \Ext^2_S(I_{Z_2}(C)/I_{Z_1},I_{Z_2}(C)) \ar[r] \ar[d]^{h^1(\mu)_P} & h^1(\dL t^*\Fbv) \ar[r] &0 \\
& \Ext^2_S(I_{Z_2}(C),I_{Z_2}(C)) & &} $$ 
where the first row is exact by the proof Proposition \ref{abs}. But since $$h^1(\mu)_P \circ h^1(\theta^\vee)_P=0,$$ the surjection $h^1(\mu)_P$ factors through $h^1(\dL t^*\Fbv)$. Therefore, by basechange again there exists a surjection $h^1(\Fbv) \to  \O_{\S{n_1,n_2}_\beta}^{ p_g}.$
 
\begin{prop} \label{vanish}
If the condition \eqref{condition} is satisfied and $p_g(S)>0$, then, $$[\S{n_1,n_2}_\beta]^{\vir}=0.$$
\end{prop}
\begin{proof}
Under the assumptions of the proposition, we showed above that the obstruction sheaf admits a surjection $$h^1(\Fbv) \to  \O_{\S{n_1,n_2}_\beta}^{ p_g}\xrightarrow{[1\;\dots \;1]}  \O_{\S{n_1,n_2}_\beta},$$ and hence the associated virtual class vanishes by \cite[Theorem 1.1]{KL13}.
\end{proof}

\begin{defi} \label{defired} The map $h^1(\mu)$ induces the morphism in derived category $$\Fbv\to h^1(\Fbv)[-1]\to h^1(\Fbrv)[-1]\xrightarrow{h^1(\mu)} \O^{p_g}_{\S{\n}_\beta}[-1].$$ Dualizing gives a map $ \O^{p_g}_{\S{\n}_\beta}[1]\to \Fb$. Define $\Fb_{\red}$ to be its cone.
\end{defi}

We will show that under a slightly stronger condition than \eqref{condition}, $\Fb_{\red}$ gives rise to a perfect obstruction theory over $\S{n_1,n_2}_\beta$. First note that the curve class $\beta \in H^{1,1}(S)\cap H^2(S,\ZZ)$ defines an element of $H^1(\Omega_S)$ and consider the natural pairing $T_S\otimes \Omega_S\to \O_S$. This condition is\footnote{This is condition (3) in \cite{KT14}.} 
\begin{equation} \label{scond} H^1(T_S)\xrightarrow{*\cup \beta} H^2(T_S\otimes \Omega_S)\to  H^2(\O_S)\quad \text{is surjective.}\end{equation}

To show $\Fb_{\red}$ gives rise to a perfect obstruction theory, we use the beautiful idea of \cite{KT14}. We sketch their method here and make some necessary changes; the reader can find the missing detail in \cite{KT14}. $S$ is embedded as the central fiber of an algebraic twistor family  $\cS\to B$, where $B$ is a first order Artinian neighborhood of the origin in a certain $p_g$-dimensional family of the first order deformations of $S$ \footnote{To simplify notation, we have used $\cS$ instead of $\cS_B$ that was used in \cite{KT14}.}. Explicitly, let  $$V\subset H^1(T_S)$$ be  a subspace over which $*\cup \beta$ in \eqref{scond} restricts to an isomorphism, and let $\mathfrak m$ denote the maximal ideal at the origin $0 \in H^1(T_S)$. Then, $$B:=\text{Spec} (\O_V/\mathfrak m^2),$$ and $\cS$ is the restriction of a tautological flat family of surfaces with Kodaira-Spencer class the identity in $$H^1(T_S)^*\otimes
H^1(T_S)\cong \Ext^1(\Omega_S,\O_S\otimes H^1(T_S)).$$  The Zarisiki tangent space $T_B$ is naturally identified with $V$.  By \cite[Lemma 2.1]{KT14}, $\cS$ is transversal to the Noether-Lefschetz locus of the $(1,1)$-class $\beta$, and as a result, $\beta$ does not deform outside of the central fiber of the family. Using this fact, as in \cite[Proposition 2.3]{KT14}, one can show that \begin{equation}\label{isomnest}\S{n_1,n_2}_\beta \cong (\cS/B)^{[n_1,n_2]}_\beta,\end{equation} where the right hand side is the relative nested Hilbert scheme of the family $\cS\to B$. We use the same symbols $$\Phi: \i{n_1}\to \i{n_2}_\beta$$ as before to denote the universal objects over $\cS\times_B  (\cS/B)^{[n_1,n_2]}_\beta$, and we let $\pi$ be the projection to the second factor of $\cS\times_B  (\cS/B)^{[n_1,n_2]}_\beta$. The arguments of Section \ref{2-step} can be adapted with no changes to prove that $$\cone \bigg( \left[\dR\hom_\pi \big(\i{n_1},\i{n_1}\big)\oplus \dR\hom_\pi\big(\i{n_2},\i{n_2}\big)\right]_0 \to \dR \hom_\pi \big(\i{n_1}, \i{n_2}_\beta\big)\bigg)$$ is the virtual tangent bundle of a perfect $B$-relative obstruction theory $ \Gb_{\rel}\to \LL_{(\cS/B)^{[n_1,n_2]}_\beta/B}$, and $$\Gb:=\cone\left( \Gb_{\rel}\to   \Omega_{B}[1]\right)[-1]\to \LL_{(\cS/B)^{[n_1,n_2]}_\beta}$$ is the associated absolute perfect obstruction theory. By the definitions of $\Fb$ and $\Gb_{\rel}$, and the isomorphism \eqref{isomnest}, we see that $\Fb\cong \Gb_{\rel}$. Now we claim that the composition $$\Gb\to \Gb_{\rel}\cong \Fb \to \Fb_{\red}$$ is an isomorphism. By the definitions of $\Gb_{\rel}$ and $\Fb_{\red}$, to prove the claim, it suffices to show that \begin{equation}\label{chain} \O^{p_g}_{\S{\n}_\beta} \to  \Fb[-1]\cong\Gb_{\rel}[-1]\to \Omega_B\end{equation} is an isomorphism. By the Nakayama lemma we may check this at a closed point $P=(Z_1,C,Z_2,\phi)\in \S{n_1,n_2}_\beta$. Define the reduced Atiyah class corresponding to $P$ as follows. Consider the natural homomorphisms of sheaf of graded algebras on $S$ $$\O_{\spec \C}\to \O_{S}\oplus I_{Z_1}\xrightarrow{} \O_{S}\oplus I_{Z_2}(C).$$ The degree 1 part of the transitivity triangle (see item 1 in Subsection \ref{notcon}) associated to the graded cotangent complexes gives the first arrow in $$\coker(\phi)\to k^1\big(\mathbb{L}^{\bullet,\text{gr}}_{\O_S\oplus I_{Z_1}}\otimes (\O_S\oplus I_{Z_2(C)})\big)[1] \xrightarrow{\text{project}} \Omega_S\otimes I_2(C)[1].$$  Define $\at_{\red}(\phi)$ to be the composition of these two arrows.

  After dualizing and using the identifications above the pullback of \eqref{chain} to $P$ becomes 
\begin{align*}T_B=V\subset H^1(T_S)\xrightarrow{ \at_{\red}(\phi)}&\Ext^2(\coker (\phi),I_{Z_2}(C))\\ \xrightarrow{h^1(\mu)_P}&\Ext^2(I_{Z_2}(C),I_{Z_2}(C))\xrightarrow{\tr}H^2(\O_S).\end{align*}
Here as in \cite{KT14}, one needs to use a similar argument as \cite[Proposition 13]{MPT10} to deduce that the composition of $\Gb_{\rel}\to \LL_{(\cS/B)^{[n_1,n_2]}_\beta/B}$ and the Kodaira-Spencer map $\LL_{(\cS/B)^{[n_1,n_2]}_\beta/B}\to \Omega_B[1]$ for $(\cS/B)^{[n_1,n_2]}_\beta$ coincides with the cup product of $\at_{\red}(\phi)$ and the Kodaira-Spencer class for $S$. Similar to [ibid], this is achieved by relating the reduced Atiyah class of $\cS\times_B  (\cS/B)^{[n_1,n_2]}_\beta$ arising (as in Lemma \ref{atred}) from the transitivity triangle  associated to the natural maps of sheaves of graded algebras $$\O_B\to \O_{\cS\times_B  (\cS/B)^{[n_1,n_2]}_\beta}\oplus \i{n_1}\to \O_{\cS\times_B  (\cS/B)^{[n_1,n_2]}_\beta}\oplus \i{n_2}_\beta$$ to the reduced Atiyah classes of each factors.
\begin{lem} $h^1(\mu)_P\circ\at_{\red}(\phi)=\at(I_{Z_2}(C)),$ where $$\at(I_{Z_2}(C))\in \Ext^1(I_{Z_2}(C), \Omega_S\otimes I_{Z_2}(C))$$ is the usual Atiyah class.  \end{lem} 
\begin{proof}
Consider the commutative diagram of sheaves of graded algebras with all unlabelled arrows are the obvious natural maps:
$$
 \xymatrix@=2em{\C \ar[r] & \O_S\oplus I_1 \ar[r]^-{(\id, \phi)} & \O_{S} \oplus I_2(C)\\
 \C \ar[r] \ar@{=}[u] & \O_{S}  \ar[r] \ar[u]& \O_{S}\oplus I_2(C).\ar@{=}[u]
}$$
Taking the degree 1 part of the the transitivity triangles of the rows followed by a projection as  in the definition of $\at_{\red}(\phi)$ above we get the commutative diagram
$$
 \xymatrix@=3em{\coker(\phi) \ar[r]^-{\at_{\red}(\phi)} & \Omega_S \otimes I_2(C)[1]\\
 I_2(C) \ar[r]^-{\at(I_2(C))} \ar[u]^-{h^1(\mu)_P} &\Omega_S \otimes I_2(C)[1] \ar@{=}[u]
}$$ proving the lemma.
\end{proof}

But by \cite[Prop 4.2]{BFl03}, $$\tr \circ \at(I_{Z_2}(C))=- *\cup \beta,$$ which by condition \eqref{scond} is an isomorphism when restricted to $V\subset H^1(T_S)$, and hence the claim is proven. We have shown

\begin{prop} \label{reduced}
If the condition \eqref{scond} is satisfied, then, $\Fb_{\red}$ is a perfect obstruction theory on $\S{n}_\beta$, and hence defines a reduced virtual fundamental class
$$[\S{n_1,n_2}_\beta]^{\vir}_{\red}\in A_{d'}(\S{n_1,n_2}_\beta), \quad d'=n_1+n_2-\frac{\beta\cdot \beta^D}{2}+p_g(S).$$
\end{prop}\qed

\subsection{Invariants} \label{invs}

Let $\cP_1,\dots, \cP_s$ be polynomials in the Chern classes of $\i{n_1}, \i{n_2}$, $\I_{-\beta}$, $\pi^* T_{\S{n_1}}$, $\pi^* T_{\S{n_2}}$, etc. cupped with the pullback of a cohomology classes from $S$ then, we can define the invariant
$$\sN_S(n_1,n_2,\beta; \cP_1,\dots, \cP_s):=\int_{[\S{n_1,n_2}_\beta]^{\vir}}\prod_{i=1}^s \pi_* \cP_i.$$
If the condition \eqref{scond} is satisfied, we can define the reduced invariants $$\sN^{\red}_S(n_1,n_2,\beta; \cP_1,\dots, \cP_s):=\int_{[\S{n_1,n_2}_\beta]^{\vir}_{\red}}\prod_{i=1}^s \pi_* \cP_i.$$

Let $u:=c_1(\O(\Z_\beta))|_{x\times \S{n_1,n_2}_\beta}$, where $x\in S$ is an arbitrary closed point.
Define $$P_S(n_1,n_2,\beta):=\det_*\left(\sum_{i\ge 0} [\S{n_1,n_2}_\beta]^{\vir} \cap u^i\right) \in H^*(\pic(S)).$$

In Section \ref{special} we will see that virtual classes $[S_\beta]^{\vir}$, $[\S{0,n_2}_\beta]^{\vir}_{\red}$ coincide respectively with the virtual classes constructed in \cite{DKO07} and \cite{KT14}. Therefore, by suitable choices of the integrands $\pi_*\cP_i$, the invariants $\sN^{\red}_S(0,n_2,\beta;  \cP_1,\dots, \cP_s)$ recover the reduced stable pair invariants of \cite{KT14}. Similarly, the invariants $P_S(0,0,\beta)$ recover Poincar\'e invariants of \cite{DKO07}. In \cite{GSY17b}, we express certain contributions to the reduced localized DT invariants of $S$ in terms of the invariants $\sN_S(n_1,n_2,\beta;\cP)$ by making suitable choices of the integrand $\cP$.\footnote{To clarify the potential confusion for the reader, we emphasize that the reduced localized DT invariants of \cite{GSY17b} uses the non-reduced virtual class $[\S{n_1,n_2}_\beta]^{\vir}$.} We will study some of the invariants $\sN_S(n_1,n_2,0;\cP)$ in Sections \ref{sec:nestpoints} and \ref{sec:co}. 

\subsection{Proof of Theorem \ref{thm1} (Proposition \ref{abs-r})} \label{r-step} In Section \ref{2-step} we proved Proposition \ref{abs-r} in the case $r=2$. We now use induction on $r$ to prove the theorem in general. 
For the simplicity of the notation, we show in detail how the result of Section \ref{2-step} can be used to prove Proposition \ref{abs-r} in the case $r=3$. Other induction steps are completely similar and hence are omitted.

Suppose that  $\n:=n_{1}, n_{2}, n_{3}$ is a sequence of nonnegative integers, and $\betab:=\beta_1,\beta_{2}$ is a sequence of effective curve classes in $H^2(S,\ZZ)$. Define $\n':=n_{1}, n_{2}$. Our goal is to prove the expression in Proposition \ref{abs-r} for $r=3$ is a perfect obstruction theory.
 
Consider the chain of natural forgetful morphisms and the associated exact triangle of cotangent complexes
\begin{equation} \label{forget}\S{\n}_{\betab}\xrightarrow{f_2} \S{\n'}_{\beta_1}\xrightarrow{f_1} \S{n_1},\quad \quad \LL_{f_2}[-1]\xrightarrow{j_2} \dL f_2^* (\LL_{f_1})\xrightarrow{j_1} \LL_f \xrightarrow{j_3} \LL_{f_2}\end{equation} where $f:=f_1\circ f_2=\pts \circ \pr_1$, using the notation at the beginning of Section \ref{sec:general}.

Proposition \ref{rel}, provides the relative perfect obstruction theory for the morphism $f_1$, that we denote by \begin{equation} \label{rel2-1}\Fb_{f_1}\xrightarrow{\alpha_1} \LL_{f_1}.\end{equation}

\begin{lem} \label{rel3-2}
There exists a relative perfect obstruction theory $\Fb_{f_2}\xrightarrow{\alpha_2} \LL_{f_2}$, where
\begin{equation}\label{Ff2}\Fb_{f_2}=\dR\hom_\pi \big(\i{n_3}_{\beta_2}, \coker(\Phi_2)\otimes \omega_S\big )[1]  \end{equation}

\end{lem}

\begin{proof} The proof is along the line of the proof of Proposition \ref{rel} (see Step 2 of that proof for the corresponding expression in RHS of \eqref{Ff2}). This time the obstruction theory is obtained by the deformation/obstruction theory theory of the universal map $$\Phi_2:\i{n_2}\to \i{n_3}_{\beta_2}$$ while the data $(\i{n_1}, \i{n_2}, \Z_{\beta_1}, \Phi_1)$ is kept fixed.
\end{proof}

\begin{lem} \label{commut}
The complexes $\Fb_{f_1}$ and $\Fb_{f_2}$ fit into the following commutative diagram:

\begin{equation}\label{diagcomm}
\xymatrix{
\Fb_{f_2}[-1] \ar^{\alpha_2[-1]}[d]\ar[r]^r &\dL f_2^*( \Fb_{f_1}) \ar^{f_2^*( \alpha_1)}[d] \\  \LL_{f_2}[-1] \ar[r]^{j_2} &\dL f_2^* (\LL_{f_1}).} 
\end{equation}
\end{lem}

\begin{proof}
\textbf{Step 1:} \emph{(Define the map $r$)} All the maps in diagram \eqref{diagcomm} except $r$ are already defined above (see \eqref{forget}, \eqref{rel2-1}, and Lemma \ref{rel3-2}).  By the universal properties of the Hilbert schemes and using our convention in suppressing the pullback symbols from the universal ideal sheave, we can write

\begin{equation}\label{Ff1}\dL f_2^*( \Fb_{f_1})\cong \dR\hom_\pi \big(\i{n_2}_{\beta_1}, \coker(\Phi_1)\otimes \omega_S[1]\big )\end{equation}

Twisting by $\O(\Z_{\beta_1})$, we get $$\Phi_2(\Z_{\beta_1}):\i{n_2}_{\beta_1}\to \i{n_3}_{\beta_2}(\Z_{\beta_1}),$$ and hence \eqref{Ff2} can be written as
\begin{equation}\label{Ff2again}\Fb_{f_2}\cong \dR\hom_\pi \big(\i{n_3}_{\beta_2}(\Z_{\beta_1}), \coker(\Phi_2(\Z_{\beta_1}))\otimes \omega_S[1]\big ).\end{equation}
The chain of maps $\i{n_1}\xrightarrow{\Phi_1} \i{n_2}_{\beta_1}\xrightarrow{\Phi_2(\Z_{\beta_1})} \i{n_3}_{\beta_2}(\Z_{\beta_1}) $ induces the natural exact triangle 
\begin{equation}\label{cokers}\coker(\Phi_2(\Z_{\beta_1})\circ \Phi_1)\xrightarrow{i_3} \coker(\Phi_2(\Z_{\beta_1}))\xrightarrow{i_2[1]} \coker( \Phi_1) [1].\end{equation}
The maps $i_2[1]$ and $\Phi_2(\Z_{\beta_1})$ induce 
\begin{align} \label{define-r}
&\Fb_{f_2}[-1]\cong \dR\hom_\pi \big(\i{n_3}_{\beta_2}(\Z_{\beta_1}), \coker(\Phi_2(\Z_{\beta_1}))\otimes \omega_S\big )\to \notag \\ 
&\dR\hom_\pi \big(\i{n_2}_{\beta_1}, \coker(\Phi_2(\Z_{\beta_1}))\otimes \omega_S\big ) \to \notag \\
&\dR\hom_\pi \big(\i{n_2}_{\beta_1}, \coker(\Phi_1)\otimes \omega_S[1]\big) \cong \dL f_2^*(\Fb_{f_1}). \end{align}
 The map $r$ in diagram \eqref{diagcomm} is then defined by composition of two maps in (\ref{define-r}).


\textbf{Step 2:} (\emph{Commutativity of diagram (\ref{diagcomm}))}  We start with the following diagram in which the columns 
are the exact triangles \eqref{cokers} and \eqref{forget}: 
\begin{equation} \label{atiyah-commut}
\xymatrix@C=85pt@R=35pt{
\coker( \Phi_1) [1]  \ar^-{(\id\otimes \Phi_2(\Z_{\beta_1}))\circ \at_{\red}(\Phi_1 )[1]}[r] &\pi^* \dL f_2^*\big(\LL_{f_1}\big)[1]\otimes \i{n_3}_{\beta_2}(\Z_{\beta_1})[1] \\
\coker(\Phi_2(\Z_{\beta_1}))  \ar^-{\at_{\red}(\Phi_2(\Z_{\beta_1}) )}[r] \ar[u]_{i_2[1]}&\pi^* \LL_{f_2}\otimes \i{n_3}_{\beta_2}(\Z_{\beta_1})[1] \ar[u]_{(\pi^*j_2[1]\otimes \id)[1]}\\
\coker(\Phi_2(\Z_{\beta_1})\circ \Phi_1)  \ar^-{\at_{\red}(\Phi_2(\Z_{\beta_1})\circ \Phi_1)}[r] \ar[u]_{i_3}& \pi^*\LL_{f}\otimes \i{n_3}_{\beta_2}(\Z_{\beta_1})[1] \ar[u]_{(\pi^*j_3\otimes \id)[1]}\\
\coker( \Phi_1) \ar[u]_{i_1}  \ar^-{(\id\otimes \Phi_2(\Z_{\beta_1}))\circ\at_{\red}(\Phi_1)}[r] & \pi^*\dL f_2^*\big(\LL_{f_1}\big)\otimes \i{n_3}_{\beta_2}(\Z_{\beta_1})[1] \ar[u]_{(\pi^*j_1\otimes \id)[1]}} \end{equation}

We prove diagram \eqref{atiyah-commut} is commutative. For this, consider the following natural commutative diagrams of sheaf of graded algebras over $S\times \S{\n}_{\betab}$ (following our convention we have suppress the symbols for pullbacks of theses via natural morphisms):
 
 $$ \xymatrix@=1em{\O_{S\times \S{n_1}} \ar[r]& \O_{S\times \S{\n}_{\betab}}\oplus \i{n_1} \ar[r]& \O_{S\times \S{\n}_{\betab}}\oplus \i{n_3}_{\beta_2}(\Z_{\beta_1})\\
 \O_{S \times \S{n_1}} \ar[r] \ar@{=}[u]& \O_{S\times \S{\n'}_{\beta_1}} \oplus \i{n_1} \ar[r] \ar[u]& \O_{S\times \S{\n'}_{\beta_1}}\oplus \i{n_2}_{\beta_1},\ar[u]&
 } $$ and  $$
 \xymatrix@=1em{\O_{{ S\times \S{\n'}_{\beta_1}}} \ar[r]& \O_{S\times \S{\n}_{\betab}}\oplus \i{n_2}_{\beta_1} \ar[r]& \O_{S\times \S{\n}_{\betab}}\oplus \i{n_3}_{\beta_2}(\Z_{\beta_1})\\
 \O_{S \times \S{n_1}} \ar[r] \ar[u]& \O_{S\times \S{\n}_{\betab}} \oplus \i{n_1} \ar[r] \ar[u]& \O_{S\times \S{\n}_{\betab}}\oplus \i{n_3}_{\beta_2}(\Z_{\beta_1}).\ar@{=}[u]&
 }$$
Applying $k^1(-)$ to the resulting commutative diagrams of the transitivity (see item 1 in Subsection \ref{notcon}) triangle of each row, we get the commutativity of the following two squares:

$$\xymatrix@C=1em{
\coker(\Phi_2(\Z_{\beta_1})\circ \Phi_1)  \ar[r] & \left(\pi^*\LL_{f}\otimes \i{n_3}_{\beta_2}(\Z_{\beta_1})[1]\right) \oplus \i{n_2}_{\beta_1}[1] \\
\coker( \Phi_1) \ar[u]_{i_1}  \ar[r] & \left(\pi^*\dL f_2^*\left(\LL_{f_1}\right)\otimes \i{n_3}_{\beta_2}(\Z_{\beta_1})[1]\right) \oplus \i{n_1}[1], \ar[u]_{(\pi^*j_1\otimes \id\oplus \Phi_1)[1]} }
$$ and
$$\xymatrix@C=1em{
\coker(\Phi_2(\Z_{\beta_1}))  \ar[r] &\left(\pi^* \LL_{f_2}\otimes \i{n_3}_{\beta_2}(\Z_{\beta_1})[1]\right) \oplus \i{n_3}_{\beta_2}(\Z_{\beta_1})[1]\\
\coker(\Phi_2(\Z_{\beta_1})\circ \Phi_1)  \ar[r] \ar[u]_{i_3}& \left(\pi^*\LL_{f}\otimes \i{n_3}_{\beta_2}(\Z_{\beta_1})[1]\right) \oplus  \i{n_2}_{\beta_1}[1] \ar[u]_{(\pi^*j_3\otimes \id\oplus \Phi_2(\Z_{\beta_1}))[1]}.}
$$

Now projecting to the first factors in the second columns of the last two diagrams, and using the definition of $\at_{\red}(-)$ (from the proof of Lemma \ref{atred}), we obtain the commutativity of  the bottom and middle squares of diagram \eqref{atiyah-commut}. Since in diagram \eqref{atiyah-commut} both columns are exact triangles, the commutativity of the top square follows, and hence we have proven that the whole diagram \eqref{atiyah-commut} commutes.

Recall from Step 2 in the proof of Proposition \ref{rel}, that the maps $\alpha_i: \Fb_{f_i}\to \LL_{f_i}$ are induced from the classes $\at_{\red}(\Phi_1)$ and $\at_{\red}(\Phi_2(\Z_{\beta_1}))$. Therefore, by the definition of the map $r$ in Step 1 of the proof, the commutativity of diagram (\ref{diagcomm}) is equivalent to the commutativity of the top square in diagram  \eqref{atiyah-commut} proven above, and hence the proof of lemma is complete. 
 
 \end{proof}
 
As a result of Lemma \ref{commut} we get a commutative digram   

 \begin{equation}\label{Ff}
\xymatrix{
\Fb_{f_2}[-1] \ar^{\alpha_2[-1]}[d]\ar[r]^r &\dL f_2^*( \Fb_{f_1}) \ar^{f_2^*( \alpha_1)}[d] \ar[r] & \cone(r)=:\Fb_f \ar@{-->}[d]^{\alpha_3} \\  \LL_{f_2}[-1] \ar[r]^{j_2} &\dL f_2^* (\LL_{f_1}) \ar[r]^{j_1}& \LL_{f},} 
\end{equation}  in which both rows are exact triangles (the commutativity of the right square was established in Lemma \ref{commut}). 

\begin{prop}
$\alpha_3:\Fb_f\to \LL_f$ is a relative perfect obstruction theory.
 \end{prop} 
 \begin{proof}
Since $\alpha_1$ and $\alpha_2$ are perfect obstruction theories, we know $\Fb_{f_2}[-1]$ is of perfect amplitude contained in $[0,1]$ and  $\dL f_2^*( \Fb_{f_1})$ is of perfect amplitude contained in $[-1,0]$, therefore $\Fb_f=\cone(r)$ is of perfect amplitude contained in $[-1,0]$. It remains to show that $h^0(\alpha_3)$ is an isomorphism and $h^{-1}(\alpha_3)$ is surjective. From diagram \eqref{Ff} and the fact that $\alpha_1$ and $\alpha_2$ are perfect obstruction theories, we get the following commutative diagram in which both rows are exact:
$$
\xymatrix{
h^{-1}(\Fb_{f_1}) \ar@{->>}[d] \ar[r] &h^{-1}(\Fb_{f}) \ar[d]^-{h^{-1}(\alpha_3)} \ar[r]& h^{-1}(\Fb_{f_2}) \ar@{->>}[d] \ar[r] &h^{0}(\Fb_{f_1}) \ar@{=}[d]\ar[r]& h^{0}(\Fb_{f}) \ar[d]^-{h^0(\alpha_3)} \ar[r] & h^{0}(\Fb_{f_2}) \ar@{=}[d] \ar[r] & 0\\
h^{-1}(\LL_{f_1}) \ar[r] &h^{-1}(\LL_{f}) \ar[r] & h^{-1}(\LL_{f_2}) \ar[r] &h^{0}(\LL_{f_1})\ar[r] & h^{0}(\LL_{f}) \ar[r] & h^{0}(\LL_{f_2}) \ar[r] & 0. 
}
$$
Applying 4-lemma once to the leftmost three squares and once to the rightmost three squares above prove the desired properties for $h^0(\alpha_3)$ and $h^{-1}(\alpha_3)$. 

 
 \end{proof}
 
 \begin{proof}[Proof of Proposition \ref{abs-r} (for $r=3$)] 
 First note that by construction, for $i=1, 2$,
 $$\Fbv_{f_i}=\cone\Big(\dR\hom_\pi\big(\i{n_{i+1}}, \i{n_{i+1}}\big)\xrightarrow{\Xi_i} \dR \hom_\pi\big(\i{n_i}, \i{n_{i+1}}_{\beta_i}\big)\Big).$$ 
Now define $$\Ab_i:=\dR\hom_\pi\big(\i{n_i}, \i{n_i}\big),\quad \Bb_j:= \dR\hom_\pi\big(\i{n_j}, \i{n_{j+1}}_{\beta_{j}}\big)\quad i=1,2,3, \; j=1,2,$$
and consider the following two commutative diagrams
$$\xymatrix@=2.7em{\Ab_3\ar[r] \ar@{=}[d]& \Ab_2\oplus \Ab_3 \ar[r] \ar[d]_-{\tiny \left[\begin{array}{cc}\Xi_1 & 0\\ -\Xi'_2 & \Xi_2\end{array} \right]} & \Ab_2\ar[d]_-{[\Xi_1\;-q \circ \Xi'_2]^t}\\ \Ab_3\ar[r]_-{[0\; \Xi_2]^t} & \Bb_1 \oplus \Bb_2 \ar[r]_-{\tiny \left[\begin{array}{cc}\id & 0\\ 0 & q\end{array} \right]} & \Bb_1\oplus \cone(\Xi_2)}\quad
 \xymatrix@=2.5em{\cone(\Xi_1)[-1]\ar[r] \ar[d]^{r^\vee[-1]}& \Ab_2\ar[r]^{\Xi_1} \ar[d]_-{[\Xi_1\;-q \circ \Xi'_2]^t}& \Bb_1\ar@{=}[d]\\ \cone(\Xi_2) \ar[r] & \Bb_1 \oplus \cone(\Xi_2)\ar[r]& \Bb_1}$$
in which all four rows are natural exact triangles and $q:\Bb_2\to \cone(\Xi_2)$ is the natural map. 
Taking cones of the columns of the right diagram gives $$\Fbv_f\cong \cone(r^\vee[-1])\cong \cone\big([\Xi_1\;-q \circ \Xi'_2]^t\big).$$ 
Therefore, taking cones of the columns of the left diagram 
\begin{align*}
\Fbv_f\cong  \cone\big([\Xi_1\;-q \circ \Xi'_2]^t\big)\cong \cone \Big (\Ab_2 \oplus \Ab_3\xrightarrow{\tiny \left[\begin{array}{cc}\Xi_1 & 0\\ -\Xi'_2 & \Xi_2\end{array} \right]} \Bb_1\oplus \Bb_2\Big).
\end{align*}

As in the proof of Proposition \ref{abs}, the fact that $\S{n_1}$ is nonsingular can be used to show that $$\Fb:=\cone\left( \Fb_f\to  f^* \Omega_{\S{n_1}}[1]\right)[-1]$$ is an absolute perfect obstruction theory for $\S{\n}_\betab$, and then (using the expression above for $\Fbv_{f}$) to prove that 
\begin{align*}
\Fbv\cong \cone \big([\Ab_1\oplus \Ab_2 \oplus \Ab_3]_0\to \Bb_1\oplus \Bb_2\big),
\end{align*}  where the arrow is as in Proposition  \ref{abs-r}.
 \end{proof}

\section{Special cases} \label{special} In this section, we show that the virtual fundamental classes arising from the perfect obstruction theories $\Fb$ and $\Fb_{\red}$ of Propositions \ref{abs} and \ref{reduced} specialize to several interesting and important cases such as the ones arising from the algebraic Seiberg-Witten theory and the reduced stable pair theory of surfaces. For the sake of brevity we do not try to match our perfect obstruction theories with these other cases, but rather we only match $K$-group classes of the underlying virtual tangent bundles. Since the virtual fundamental class only depends on the K-theory class of the virtual tangent bundle \cite{S04}, this is sufficient for the purpose of the following proposition:

\begin{prop}\label{thm1.2} The virtual fundamental class of Theorem \ref{thm1} recovers the following known cases:
\begin{enumerate}[1.]
\item If $\beta=0$ and $n_1=n_2=n$ then $\S{n,n}_{\beta=0}\cong \S{n}$ and $[\S{n,n}_{\beta=0}]^{\vir}=[\S{n}]$ is the fundamental class of the Hilbert scheme of $n$ points.

\item  If $\beta=0$ and $n=n_2=n_1-1$, as it is known $$\S{n+1,n}_{\beta=0}\cong \PP(\i{n}):=\operatorname{Proj} \operatorname{Sym}(\i{n})\to S\times \S{n}$$ is nonsingular \cite[Section 1.2]{L99}, then, $$[\S{n+1,n}_{\beta=0}]^{\vir}=[\S{n+1,n}_{\beta=0}]\cap c_1(H)$$ for a line bundle  $H$ on $\PP(\i{n})$.
\item If $\beta=0$ and $n_2=0$, then $\S{n,0}_{\beta=0}\cong \S{n}$ and $$[\S{n,0}_{\beta=0}]^{\vir}=(-1)^{n}[\S{n}]\cap c_{n}(\omega_S^{[n]}),$$ where $\omega_S^{[n]}$ is the rank $n$ tautological vector bundle over $\S{n}$ associated to the canonical bundle $\omega_S$ of $S$.

\item If $n_1=n_2=0$ and $\beta\neq 0$, then $\S{0,0}_\beta=S_\beta$ is the Hilbert scheme of divisors in class $\beta$, and $[\S{0,0}_\beta]^{\vir}$ and $[\S{0,0}_\beta]^{\vir}_{\red}$ (in case $\beta$ satisfies condition \eqref{scond}) coincide with the virtual cycles constructed in \cite{DKO07}.
\item If $n_1=0$ and $\beta\neq 0$, then $\S{0,n_2}_\beta$ is the relative Hilbert scheme of points on the universal divisor over $\S{0,0}_\beta$, and by \cite{PT10} is isomorphic to a moduli space of stable pairs; $[\S{0,n_2}_\beta]^{\vir}$ and $[\S{0,n_2}_\beta]^{\vir}_{\red}$ (in case $\beta$ satisfies condition \eqref{scond}) are the same as the virtual fundamental classes of \cite{KT14}.
\end{enumerate}
\end{prop}

\begin{proof}  
If $\beta=0$, the nested Hilbert scheme of points $\S{n_1\ge n_2}:=\S{n_1,n_2}_{\beta=0}$ carries a virtual fundamental class 
$$[\S{n_1\ge n_2}]^{\vir}\in A_{n_1+n_2}(\S{n_1\ge n_2}).$$ 
Note that by \cite{C98}, $\S{n_1\ge n_2}$ is nonsingular only in the following two cases: 

$\bullet$ $n_1=n_2$. In this case $\S{n_1\ge n_2}\cong \S{n_1}$ by definition, and $[\S{n_1}]^{\vir}=[\S{n_1}]$, because by Proposition \ref{rel}, $\Fbr\cong 0$, and so by Proposition \ref{abs}, $\Fb\cong \Omega_{\S{n_1}}$. This gives part 1.

$\bullet$ $n_1=n_2+1$. In this case, since $\S{n_2+1,n_2}$ is nonsingular of dimension $2n_2$ (see \cite{C98, L99}). The virtual dimension is $2n_2+1$, and hence the obstruction sheaf $H:=h^1(\Fbv)$ is an invertible sheaf. Then, we can write
$$[\S{n_2+1,n_2}]^{\vir}=[\S{n_2+1,n_2}]\cap c_{1}\left(H\right),$$ where we have used \cite[Proposition 5.6]{BF97} to write $[\S{n_1\ge n_2}]^{\vir}$ as the fundamental class capped with the Euler class of the obstruction bundle.
We can express $c_1(H)$ in terms of other classes.  We know that  $\S{n_2+1,n_2}\cong \PP(\i{n_2})$ (see \cite[Section 1.2]{L99}), so in the $K$-group of $\PP(\i{n_2})$ we can write 
$$H-T_{\PP(\i{n_2})}=[\dR\hom_{\pi}(\i{n_1},\i{n_1})\oplus \dR\hom_{\pi}(\i{n_2},\i{n_2})]_0-\dR\hom_{\pi}(\i{n_1},\i{n_2}).$$
In taking the Chern class, we can ignore the trivial terms and hence we have 
$$c_1(H)=c_1\big(T_{\PP(\i{n_2})}-T_{\S{n_1}}-T_{\S{n_2}}-\dR\hom_{\pi}(\i{n_1},\i{n_2})\big).$$ This proves part 2.


 If $n_2=0$ and $\beta=0$, then we get a perfect obstruction theory over the nonsingular Hilbert scheme of points $\S{n_1}$ that is arising from the natural obstruction theory of the Hilbert scheme. In fact in this case 

\begin{align*}\Fbv&\cong \cone \bigg( \left[\dR\hom_\pi \big(\i{n_1},\i{n_1}\big)\oplus \dR\hom_\pi\big(\O,\O\big)\right]_0 \to \dR \hom_\pi \big(\i{n_1}, \O \big)\bigg)
\\&\cong\cone \bigg( \dR\hom_\pi \big(\i{n_1},\i{n_1}\big) \to \dR \hom_\pi \big(\i{n_1}, \O \big)\bigg)
\\& \cong \dR\hom_\pi \big(\i{n_1},\O_{\Z^{[n_1]}}\big).
\end{align*} Note that $$T_{\S{n_1}}=h^0(\Fbv)\cong  \hom_\pi \big(\i{n_1},\O_{\Z^{[n_1]}}\big),\quad h^1(\Fbv)\cong  \ext^1_\pi \big(\i{n_1},\O_{\Z^{[n_1]}}\big). $$
Since $\S{n_1}$ is nonsingular of dimension $2n_1$, we see that the obstruction sheaf $h^1(\Fbv)$ is a vector bundle of rank $n_1$, and hence by  \cite[Proposition 5.6]{BF97} 

$$[\S{n_1}]^{\vir}= [\S{n_1}]\cap  c_{n_1}\left(\ext^1_\pi \big(\i{n_1},\O_{\Z_1}\big)\right).$$

We were notified by Richard Thomas that the obstruction bundle $\ext^1_\pi \big(\i{n_1},\O_{\Z_1}\big)$ can be identified with the dual of the tautological bundle $\omega_S^{[n]}:=\pi_*\left(\omega_S|_{\O_{\Z_1}}\right)$. This can be seen by applying $\hom\big(\O_{\Z_1},-\big)$ to the short exact sequence $$0\to \i{n_1}\otimes  \omega_S\to \omega_S\to \omega_S|_{\O_{\Z_1}}\to 0 $$ over $S\times \S{n_1}$ to get the isomorphism $$\omega_S|_{\O_{\Z_1}}\cong \hom \big(\O_{\Z_1},\omega_S|_{\O_{\Z_1}}\big)\cong \ext^1\big(\O_{\Z_1}, \i{n_1}\otimes \omega_S\big).$$ Now pushing forward, we prove the claim$$\omega_S^{[n]}\cong \pi_* \ext^1\big(\O_{\Z_1}, \i{n_1}\otimes \omega_S\big)\cong  \ext^1_\pi\big(\O_{\Z_1}, \i{n_1}\otimes \omega_S\big)\cong  \ext^1_\pi\big(\i{n_1},\O_{\Z_1}\big)^*,$$ where the second isomorphism is because of local to global spectral sequence (as $\Z_1$ is fiberwise 0-dimensional) and the third one is by Grothendieck-Verdier duality. This completes the proof of part 3.

 If $n_1=n_2=0$, and $\beta\neq 0$, the perfect obstruction theory $\Fbv$ on $S_\beta=\S{0,0}_\beta$ specializes to 
\begin{align*}\Fbv&\cong \cone \bigg( \left[\dR\hom_\pi \left(\O,\O\right)\oplus \dR\hom_\pi\left(\O,\O\right)\right]_0 \to \dR \hom_\pi \left(\O, \O(\Z_\beta) \right)\bigg)
\\&\cong \dR\pi_*\O_{\Z_\beta}(\Z_\beta).
\end{align*}
studied by D\"{u}rr-Kabanov-Okonek \cite{DKO07} in the course of algebraic Seiberg-Witten invariants (Poincar\'e invariants). Moreover, one can see by inspection that under condition \eqref{scond},  the $K$-group class of $(\Fb_{\red})^\vee$ coincides with the $K$-group class of the reduced virtual tangent bundle over $S_\beta$ constructed in \cite{DKO07}. This is because 
by Definition \ref{defired} in the $K$-group  $(\Fb_{\red})^{\vee}=\Fbv+\O_{S_\beta}^{p_g}$ and the same is true for the reduced virtual tangent constructed in \cite{DKO07}. This proves part 4.


Finally, if $n_1=0$, $n_2\neq 0$ and $\beta\neq 0$, then by \cite[Prop B.8]{PT10}, $$\S{0,n_2}_\beta= \text{Hilb}^{n_2}(\Z_\beta/S_\beta) \cong P_{n_2-\beta\cdot (\beta+K_S)/2}(S,\beta),$$
where $ \text{Hilb}^{n_2}(\Z_\beta/S_\beta) $ is the relative Hilbert scheme of points on the universal curve $\Z_\beta$, and $P_{-}(S,-)$ is the moduli space of stable pairs on $S$.
Let $\O_{S\times P}\to \FF$ be the universal stable pair over $S\times P_{n_2-\beta\cdot (\beta+K_S)/2}(S,\beta)$, and let $\IIb$  be the associated complex. In this case $\Fbv$ is given by
\begin{align*}\Fbv&\cong \cone \bigg( \left[\dR\hom_\pi \left(\O,\O\right)\oplus \dR\hom_\pi\big(\i{n_2},\i{n_2}\big)\right]_0 \to \dR \hom_\pi \big(\O, \i{n_2}_\beta \big)\bigg)
\\&\cong \cone \bigg( \dR\hom_\pi \big(\i{n_2},\i{n_2}\big) \to \dR \hom_\pi \big(\O(-\Z_\beta), \i{n_2}\big)\bigg)
\\&\cong \dR\hom_\pi \big(\I_{\Z^{[n_2]}\subset \Z_\beta},\i{n_2}\big)[1]
\cong \dR\hom_\pi \big(\IIb,\FF\big).
\end{align*} 
Here by $\I_{\Z^{[n_2]}\subset \Z_\beta}$ we mean the push-forward of the ideal sheaf of $\Z^{[n_2]}$ as a subscheme of  $\Z_\beta \subset S\times \S{0,n_2}_\beta$. The last isomorphism above follows from (91) in \cite{KT14}. We have shown that in this case $\Fbv$ coincides with virtual tangent bundle of the stable pair moduli space $P_{n_2-\beta\cdot (\beta+K_S)/2}(S,\beta)$. Moreover, by the same reasoning given for the proof of part 4, one can see by inspection that under condition \eqref{scond},  the $K$-group class of $(\Fb_{\red})^\vee$ coincides with the $K$-group class of the reduced virtual tangent bundle over  $P_{n_2-\beta\cdot (\beta+K_S)/2}(S,\beta)$ constructed in \cite{KT14}. This finishes the proof of part 5.

\end{proof}

%
%

\section{Punctual nested Hilbert schemes} \label{sec:nestpoints}
We will discuss a few tools for evaluating the virtual fundamental class $[\S{n_1\ge n_2}]^{\vir}$ constructed in Corollary \ref{virclass}.  We first develop a localization formula \eqref{virtan} in the case that $S$ is toric along the lines of \cite{MNOP06}. When $S$ is toric we express $\iota_*[\S{n_1\ge n_2}]^{\vir}$ as the top Chern class of a vector bundle over the product of Hilbert schemes $\S{n_1}\times \S{n_2}$ (see Proposition \ref{nestprodfano}). We have not been able to prove such a formula for general projective surfaces. Instead, we prove a weaker statement for general projective surfaces in which the integral of certain cohomology classes against $[\S{n_1\ge n_2}]^{\vir}$ is expressed in terms of integrals over $\S{n_1}\times \S{n_2}$. This is done by using degeneration and the double point relations (see Corollary \ref{cor:znzp}, Proposition \ref{genznzp}). Such integrals arise in all the applications that we have in mind, particularly, they are related to the localized DT invariants of $S$ discussed in \cite{GSY17b}. In Section \ref{sec:co} we express some of these integrals against $[\S{n_1\ge n_2}]^{\vir}$ in terms of Carlsson-Okounkov's vertex operators and as a result obtain explicit product formulas for their generating series. 

Recall that 
$$\S{n_1\ge n_2}=\left\{(Z_1,Z_2) \mid Z_i \in \S{n_i},  \;\;  Z_1 \supseteq Z_2 \right\}\subset \S{n_1} \times \S{n_2}.$$ For simplicity in this section, we denote by $I_i$ the ideal sheaf $I_{Z_i}$ of $Z_i$. Hence for any closed point $(Z_1,Z_2) \in \S{n_1\ge n_2}$ we have $I_1 \subseteq I_2$. Sometimes, we denote the closed point above by the pair $(I_1,I_2)$, or by $I_1\subseteq I_2$, when we want to emphasize the inclusion of subschemes. As before, we have the universal objects over $S\times \S{n_1\ge n_2}$:
$$\Phi\colon \i{n_1}\hookrightarrow \i{n_{2}}.$$ 

We will use the following simple lemma in Section \ref{sec:toric}:
\begin{lem} \label{I1I2} \begin{enumerate}[1.]
\item If $(I_1\subseteq I_2) \in \S{n_1,n_2}$ is a closed point, then $$\Hom_S(I_1,I_2)=\Hom_S(I_1,I_1)=\Hom_S(I_2,I_2)=H^0(\O_S)\cong \C.$$ \item If $(I_1,I_2)\in \S{n_1}\times \S{n_2} \setminus \S{n_1,n_2}$ is a closed point then $\Hom_S(I_1,I_2)=0$.
\item If $p_g(S)=0$ and if $(I_1,I_2)\in \S{n_1}\times \S{n_2}$ is a closed point then $\Ext^2_S(I_i,I_j)=0$.
\end{enumerate}
\end{lem}
\begin{proof}
Applying the functor $\Hom(I_1,-)$ to the short exact sequence $0\to I_2\to \O_S\to \O_{Z_2}\to 0$, we get the exact sequence $$0\to \Hom(I_1,I_2)\subseteq \Hom(I_1,\O_S)\cong H^0(\O_S)=\C\xrightarrow{u} \Hom(I_1,\O_{Z_2}),$$ where $u$ composes any map $I_1\to \O_S$ with the natural map $\O_S\to \O_{Z_2}$.  In part 1 the inclusion $I_1\subseteq I_2$ gives a nonzero element of $\Hom(I_1,I_2)$ and hence the claim follows. In part 2, $u(I_1\subset \O_S)\neq 0$ because $I_1\not \subset I_2$, and so the claim is proven. For part 3, applying the functor $\Hom(I_j,-)$ to the short exact sequence $0\to I_i\otimes \omega_S\to \omega_S\to \O_{Z_i}\to 0$, we get $$ \Hom(I_j,I_i\otimes \omega_S)\subseteq \Hom_S(I_j,\omega_S)=H^0(\omega_S)=0,$$ and so the claim follows by Serre duality.
\end{proof}
As will become clear shortly, the following $K$-group element plays an important role in the rest of the paper:
\begin{defi} \label{virbdl} For any line bundles $M$ on $S$, let $\sE_M^{n_1,n_2} \in K(\S{n_1}\times \S{n_2})$  be the element of rank $n_1+n_2$ defined by $$\sE_M^{n_1,n_2}:=\left[\dR\pi_*p^*M \right]-\left[\dR\hom_{\pi}(\i{n_1},\i{n_2}\otimes p^*M)\right],$$ where $p$ and  $\pi$ are respectively the projections from $S\times \S{n_1}\times \S{n_2}$ to the first and the product of the last two factors. 
Similarly, we define the twisted tangent bundle as the rank $2n_i$ element of $K(\S{n_1}\times \S{n_2})$ 
$$\sT^M_{\S{n_i}}:=\left[\dR\pi_*p^*M \right]-\left[\dR\hom_{\pi}(\i{n_i},\i{n_i}\otimes p^*M)\right].$$ Note that $\sT^{\O_S}_{\S{n_i}}$ is the class of (pullback of)   the usual tangent bundle of $\S{n_i}$. If $M=\O_S$, we sometimes drop it from the notation. 
\end{defi}

\subsection{Toric surfaces} \label{sec:toric} 
Let $(\C^2)^{[n_1\ge n_2]}$ be the nested Hilbert scheme of points on $\C^2=\text{Spec}( R)$, where $R=\C[x_1,x_2]$. The $2$-dimensional torus $\dT$ acts on $\C^2$. We denote by $t_{1},t_{2}$ the torus characters, such that the tangent space at $0\in \C^2$ has the $\dT$-character $t_1^{-1}+ t_2^{-1}$. The $\dT$-fixed set $$\ct{n_1\ge n_2}\subset \ct{n_1}\times \ct{n_2}$$ is isolated, and is given by the inclusion of the monomial ideals $I_1\subseteq I_2$ or equivalently the  corresponding nested partitions  $\mu' \subseteq \mu$.  By Proposition \ref{abs} and Lemma \ref{nofixed}, the virtual tangent space at the $\dT$-fixed point $I_1\subseteq I_2$ is given by\footnote{This is obtained by taking the derived restriction of the complex $\Fb$ to the point $I_1\subseteq I_2$, and then taking the $K$-group class of the resulting complex. Also note that by slightly modifying the proof of part 1 of Lemma \ref{I1I2},  $\Hom(I_1,I_1)=\Hom(I_1,I_2)=\Hom(I_2,I_2)=R$.}
\begin{equation} \label{Tvirchar} \cT^{\vir}_{I_1\subseteq I_2}=-\chi(I_1,I_1)-\chi(I_2,I_2)+\chi(I_1,I_2)+\chi(R,R),\end{equation} where $\chi(-,-)=\sum_{i=0}^2(-1)^i\Ext_R^i(-,-)$.
By the exact method as in \cite[Section 4.6]{MNOP06} using Taylor resolutions and \v{C}ech complexes, the $\dT$-representation of $\cT^{\vir}_{I_1\subseteq I_2}$ can be explicitly written down as a Laurent polynomial in $t_1$ and $t_2$. For the $\dT$-fixed 0-dimensional subschemes  $Z_2\subseteq Z_1 \subset \C^2$ corresponding to the monomial ideals $I_1\subseteq I_2$ define $$\cZ_1:=\sum_{(k_1,k_2)\in \mu}t_1^{k_1}t_2^{k_2}=\frac{1-P_1(t_1,t_2)}{(1-t_1)(1-t_2)},\quad \quad \cZ_2:=\sum_{(k_1,k_2)\in \mu'}t_1^{k_1}t_2^{k_2}=\frac{1-P_2(t_1,t_2)}{(1-t_1)(1-t_2)}.$$ Here, $P_1$, $P_2$ are the Poincar\'e polynomials associated to the monomial ideals $I_1$ and $I_2$ (defined using their Taylor resolutions, see \cite[Section 4.7]{MNOP06}), respectively. Also, define $$\overline{P_i}:=P_i(t_1^{-1},t_2^{-1}),\quad \quad \overline{\cZ}_i:=\cZ_i(t_1^{-1},t_2^{-1}), \quad \quad i=1, 2.$$ Putting these expressions into \eqref{Tvirchar} and simplifying, we get
\begin{align} \label{virtan}\tr_{\cT^{\vir}_{I_1\subseteq I_2}}&=\frac{-\overline{ P}_1P_1-\overline{ P}_2P_2+\overline{ P}_1P_2+1}{(1-t_1)(1-t_2)}\\ \notag &=\cZ_1+\frac{\overline{\cZ}_2}{t_1t_2}+\left(\overline{\cZ}_1\cdot\cZ_2-\overline{\cZ}_1\cdot \cZ_1-\overline{\cZ}_2\cdot \cZ_2\right)\frac{(1-t_{1})(1-t_{2})}{t_{1}t_{2}}.\end{align}

Now if $S$ is a toric surface, then the set of $\dT$-fixed points of $\S{n_1\ge n_2}\subset \S{n_1}\times \S{n_2}$ is again isolated (Lemma \ref{nofixed}), and the $\dT$-character of the virtual tangent space at any fixed point is obtained by summing over the expression \eqref{virtan} for all the $\dT$-invariant open subsets of $S$. This finishes the proof of Theorem \ref{thm1.7}.

\begin{lem} \label{nofixed} Suppose that $S$ is a nonsingular projective toric surface, and  $Z_2\subseteq Z_1$ is a $\dT$-fixed point of $\S{n_1\ge n_2}$, then $\Ext^2_S(I_1,I_1)=\Ext^2_S(I_2,I_2)=\Ext^2_S(I_1,I_2)=0$, the $\dT$-representations  $$\Ext^1_S(I_1,I_1), \quad \Ext^1_S(I_2,I_2), \quad \Ext^1_S(I_1,I_2)$$  contain no trivial sub-representations.

\end{lem}

\begin{proof} The vanishings in the lemma follow from the fact that $p_g(S)=0$ for toric surfaces, and part 3 of Lemma \ref{I1I2}.
For any fixed point $\alpha\in S$, let $U_{\alpha}\cong \C^2$ be the $\dT$-invariant open neighborhood of $\alpha$, and let $I_{i,\alpha}:=I_i|_{\alpha}$, and $\O_{i,\alpha}:=\O_{Z_i}|_{\alpha}$. 
By \cite[Lemma 3.2]{ES87}, $\Hom_{U_\alpha}(I_{i,\alpha},\O_{i,\alpha})$ contains no trivial subrepresentations. Therefore, 
$$\Ext^1_S(I_i,I_i) \cong \Hom_S(I_i,\O_{Z_i})=\bigoplus_\alpha  \Hom_{U_\alpha}(I_{i,\alpha},\O_{i,\alpha})$$ contains no trivial representations either (in the first isomorphism we used the vanishing $H^1(\O_S)=0$ for toric surfaces). 

Next, applying  $\Hom_S(I_{1},-)$ to the natural short exact sequence $0\to I_2\to \O_S\to \O_{Z_2}\to 0$, we obtain the exact sequence \begin{equation}\label{homI1}\Hom_S(I_{1}, \O_{Z_2})\to \Ext^{1}_S(I_{1}, I_{2})\to \Ext^{1}_S(I_{1}, \O_S).\end{equation}To finish the proof it suffices to show that the 1st and the 3rd terms in \eqref{homI1} contain no trivial representations. The claim for the 1st term in \eqref{homI1} follows from the fact that for each $\alpha$, $\dT$ acts with different weights on the $\C$-basis elements for $I_{1,\alpha}$ and $\O_{2,\alpha}$ that are given by the monomials (because of the inclusion $I_{1,\alpha} \subseteq I_{2,\alpha}$). The claim for the 3rd term in \eqref{homI1} also follows because, 
applying $\Hom_S(-,\O_S)$  to the natural short exact sequence $0\to I_1\to \O_S\to \O_{Z_1}\to 0$, and using equivariant Serre duality, we get $$\Ext^{1}_S(I_{1}, \O_S)\cong \Ext^2_S(\O_{Z_1},\O_S)\cong H^0(\O_{Z_1}\otimes \omega_S)^*.$$ But since $Z_1$ is zero dimensional and $\dT$-fixed  $$H^0( \O_{Z_1}\otimes \omega_S)=\bigoplus_\alpha H^0(U_\alpha, \O_{Z_1}\otimes \omega_{S}).$$ For each $\alpha$, let $\mu_{\alpha}$ be the partition corresponding to $Z_1|_{U_\alpha}$, and suppose that the $\dT$-character of $T_\alpha S$ is $t_{1}^{-1}+ t_2^{-1}$ for some $\dT$-characters $t_1$ and $t_2$, then, the fiber of $\omega_{S}$ at $\alpha$ has the $\dT$-character $t_1t_2$, and therefore,  $$H^0(U_\alpha, \O_{Z_1}\otimes \omega_{S})=t_1t_2\sum_{(k_1,k_2)\in \mu_{\alpha}} t_1^{k_1}t_2^{k_2}$$ has no trivial representations. 
\end{proof}

\subsection{ Proof of Theorem \ref{thm2}}
Suppose that $S$ is a toric surface, and $(I_1,I_2)\in \S{n_1}\times \S{n_2}$ is a closed point. By Lemma \ref{nofixed}  \begin{equation} \label{fano}\Ext^2_S(I_i,I_j)=0.\end{equation} Therefore by basechange, the sheaves $$\ext^1_\pi(\i{n_1},\i{n_1}), \quad \ext^1_\pi(\i{n_2},\i{n_2}),\quad \ext^1_\pi(\i{n_1},\i{n_2})$$ are vector bundles over $\S{n_1\ge n_2}$ of ranks $2n_1, 2n_2, n_1+n_2$, respectively. Moreover,  the virtual tangent bundle of Proposition \ref{abs}, simplifies to the 2-term complex
\begin{equation} \label{2term}\Fbv=\left\{\ext^1_\pi(\i{n_1},\i{n_1})\oplus \ext^1_\pi(\i{n_2},\i{n_2})\to \ext^1_\pi(\i{n_1},\i{n_2})\right\}.\end{equation} Recall that the rank of $\Fbv$ is equal to $n_1+n_2$, and recall the  $K$-group elements  $\sE^{n_1,n_2}$ and $\sE^{n_1,n_2}|_{(I_1,I_2)}$ of ranks $n_1+n_2$ from Definition \ref{virbdl}. 
We have \begin{equation} \label{fiber} \sE^{n_1,n_2}|_{(I_1,I_2)}=\begin{cases}\Ext^1_S(I_1,I_2) & (I_1,I_2) \in \S{n_1\ge n_2}, \\ H^0(\O_S)\oplus\Ext^1_S(I_1,I_2) & (I_1,I_2) \not \in \S{n_1\ge n_2}.\end{cases}\end{equation}  This is true because of basechange, the vanishing \eqref{fano}, the vanishing $H^1(\O_S)=H^2(\O_S)=0$, and that by Lemma \ref{I1I2}, $$\Hom_S(I_1,I_2)=\begin{cases} H^0(\O_S) & I_1\subseteq I_2, \\ 0 & I_1\not \subseteq I_2.\end{cases}$$ Note that this is consistent with the fact that the dimension of $\Ext^1(I_1,I_2)$ jumps by 1 on $\S{n_1,n_2}\subset \S{n_1}\times \S{n_2}$, and that the dimension of $\sE^{n_1,n_2}|_{(I_1,I_2)}$ is constant over $\S{n_1}\times \S{n_2}$.

Now we are ready to express the main result of this section relating the push forward of $[\S{n_1\ge n_2}]^{\vir}$ to the products of the fundamental classes of Hilbert scheme of points.
The following proposition proves Theorem \ref{thm2}.
\begin{prop}\label{nestprodfano} 
Suppose that $S$ is a nonsingular projective toric surface, then, 
$$\iota_*[\S{n_1\ge n_2}]^{\vir}=c_{n_1+n_2}(\sE^{n_1,n_2}) \cap [\S{n_1}\times \S{n_2}],$$ where $\iota$ is the natural inclusion   $\S{n_1\ge n_2}\hookrightarrow \S{n_1}\times \S{n_2}$.
\end{prop}

\begin{proof} 
Let $i$ and $j$ be inclusion of the fixed point set in $\S{n_1\ge n_2}$ and $\S{n_1}\times \S{n_2}$, respectively. By \eqref{2term} and Lemma \ref{nofixed}, the virtual localization formula (see \cite{GP99}) gives
\begin{align*}[\S{n_1\ge n_2}]^{\vir}=&\sum_{(I_1 \subseteq I_2)\in \ST{n_1\ge n_2}}\frac{i_*[(I_1\subseteq I_2)]}{e(\cT^{\vir}_{I_1\subseteq I_2})}\\=&\sum_{(I_1\subseteq I_2) \in \ST{n_1\ge n_2}}\frac{e(\Ext^1_S(I_1,I_2))}{e(\Ext^1_S(I_1,I_1))e(\Ext^1_S(I_2,I_2))}i_*[(I_1\subseteq I_2)],\end{align*}
 where the sum is over the isolated  $\dT$-fixed points, and $e(-)$ indicates the equivariant Euler class. By Lemma \ref{nofixed}, the coefficient of $i_*[(I_1\subseteq I_2)]$ in the last sum is the product of the pure nontrivial  torus weights.
On the other hand, by Lemma \ref{nofixed} and the Atiyah-Bott localization formula 
 \begin{align*}&c_{n_1+n_2}(\sE^{n_1,n_2}) \cap [\S{n_1}\times \S{n_2}]=\sum_{(I_1, I_2)\in \ST{n_1}\times \ST{n_2}}\frac{e(\sE^{n_1,n_2}|_{(I_1,I_2)})}{e(T_{(I_1, I_2)}(\S{n_1}\times \S{n_2}))}j_*[(I_1, I_2)]\\&=\sum_{(I_1, I_2) \in \ST{n_1}\times \ST{n_2}}\frac{e(\sE^{n_1,n_2}|_{(I_1,I_2)})}{e(\Ext^1_S(I_1,I_1))e(\Ext^1_S(I_2,I_2))}j_*[(I_1, I_2)]\\
 &=\sum_{I_1\subseteq I_2 \in \ST{n_1\ge n_2}}\frac{e(\Ext^1_S(I_1,I_2))}{e(\Ext^1_S(I_1,I_1))e(\Ext^1_S(I_2,I_2))}\iota_*\circ i_*[I_1\subseteq I_2], \end{align*}
 where the last equality is because  of \eqref{fiber}, and the fact that since $H^0(\O_S)\cong \C$ is the trivial $\dT$-representation, we have $e(H^0(\O_S))=0$. The proposition is proven by comparing the outcomes of both localization formulas above, and taking the non-equivariant limit at the end.
 \end{proof}

\subsection{Relative nested Hilbert schemes} \label{sec:relative}

In this section we sketch how the degeneration formula of Li and Wu can be applied to the case of nested Hilbert scheme of points. Let $(S,D)$ be a pair of nonsingular projective surface and a nonsingular effective divisor. Li and Wu \cite{LW15} introduced the notion of a \emph{stable relative} ideal sheaf.
$I\in \S{n}$ is said to be relative to $D$ if the natural map \begin{equation}\label{normality} I\otimes \O_D\to \O_S\otimes \O_D\end{equation} is injective (see also \cite{MNOPII}). This is equivalent to $\O_S/I$ having support disjoint from $D$. Relativity is an open condition in $\S{n}$. Li and Wu constructed a relative Hilbert scheme, denoted by $\SD{n}$, by considering the equivalence classes of the stable relative ideal sheaves on the $k$-step semistable models $S[k]$ for $0\le k\le n$. Let $D_0,\dots,D_{k-1}$ be the singular locus of $S[k]$ and $D_k\subset S[k]$ be the proper transform of $D$. $S[k]$ consists of $k+1$ irreducible components $\Delta_0,\dots, \Delta_k$ with $\Delta_0=S$ and $D_i=\Delta_i\cap \Delta_{i+1}$ for $i=0,\dots, k-1$. A relative ideal sheaf $I$ on $S[k]$ satisfies \eqref{normality} for $D=D_0,\dots, D_k$. Two relative ideal sheaves $I$ and $I'$ on $S[k]$ are equivalent if the quotients $\O_{S[k]}/I$ and $\O_{S[k]}/I'$ differ by an automorphism of $S[k]$ covering the identity on $\Delta_0=S$. The stability of a relative ideal sheaf means that it has finitely many automorphisms as described above.  $\SD{n}$ is a smooth proper Deligne-Mumford stack of dimension $2n$. 

 Since by the relativity condition for any relative ideal sheaf $I$, $I|_{D_j}\cong \O_{D_j}$, the generalization of Li-Wu Hilbert schemes to the set up  of the nested Hilbert schemes is straightforward. In other words, we can construct a proper Deligne-Mumford stack $\SD{n_1\ge n_2}$ as the moduli space of relative ideal sheaves $I_1$ and $I_2$ with $I_1$ stable and $I_1\subseteq I_2$\footnote{Note that if $Z_i\subset S[k]$ is the 0-dimension subscheme corresponding to $I_i$, then the number of the auto-equivalences of $Z_2\subseteq Z_1\subset S[k]$ is less than or equal to that of $Z_1\subset 
S[k]$, which is finite by the stability of $I_1$.}.
\begin{notn}
Following \cite[Secttion 2.3]{LW15}, let $\mathfrak{A}_\diamond$ be the Artin stack of expanded degenerations for the pair $(S,D)$, and let $\mathcal{S}\to \mathfrak{A}_\diamond$ be the universal family of surfaces over it. They fit into the fibered diagram $$\xymatrix{
\mathcal S \ar[d] \ar[r]& S \ar[d]\\ \mathfrak{A}_\diamond \ar[r] &\operatorname{Spec} \C.
}$$ Let $\SD{n_1\ge n_2}\to \mathfrak{A}_\diamond$ be the natural morphism; it factors through the substack $\mathfrak{A}_\diamond^{[\n]}\subset \mathfrak{A}_\diamond$ corresponding to the numerical data $\n$ (see \cite[Secttion 2.5]{LW15} for the construction of these substacks\footnote{Since we are dealing with zero dimensional subschemes their Hilbert polynomials (used in \cite{LW15}) are simply  the nonnegative integers $n_1, n_2$.}). $\mathfrak{A}_\diamond^{[\n]}$ is a smooth Artin stack of dimension 0. We use the same notation as in the absolute case to denote the  inclusion of the universal objects over   $\mathcal{S}\times_{\mathfrak{A}^{[\n]}_\diamond}\SD{\n}$:
$$0\neq \Phi: \i{n_1}\to \i{n_{2}}.$$ Let $\pi$ be the projection to the second factor of $\mathcal{S}\times_{\mathfrak{A}^{[\n]}_\diamond}\SD{\n}$, and $p$ be the projection to its first factor followed by the natural map $\mathcal{S}\to S$. \end{notn}
By the method of \cite[Section 3.9]{MPT10} and \cite{LW15}, one can see, after modifying our argument for the usual nested Hilbert schemes  (Proposition \ref{abs}), that there is a relative perfect obstruction theory $\mathcal{F}_{\rel}^\bullet\to \LL_{\SD{n_1\ge n_2}/\mathfrak{A}^{[\n]}_\diamond}$ with the relative virtual tangent bundle
\begin{align} \label{obrelative} \mathcal{F}_{\rel}^{\bullet \vee}:=  \cone \bigg( &\left[\dR\hom_\pi \big(\i{n_1},\i{n_1}\big)\oplus \dR\hom_\pi\big(\i{n_2},\i{n_2}\big)\right]_0\\& \notag \to \dR \hom_\pi \big(\i{n_1}, \i{n_2}\big)\bigg).\end{align}
Since the Artin stack $\mathfrak{A}_\diamond^{[\n]}$ is smooth of dimension 0 by \cite[Section 7]{BF97} there is a virtual fundamental class $$[\SD{n_1\ge n_2}]^{\vir}\in A_{n_1+n_2}(\SD{n_1\ge n_2})$$ associated to this relative perfect obstruction  theory.

Let $X:S  \rightsquigarrow S_0:= S_1\cup_D S_2$ be a good degeneration of the surface $S$ along $D$ over a pointed curve $(C,0)$\footnote{It means we have a nonsingular threefold $X$ over $C$, whose general fibers are isomorphic to $S$, and whose fiber over $0\in C$ is a normal crossing divisor $S_1\cup_D S_2$ consisting of two nonsingular surfaces $S_1, S_2$ glued along a nonsingular divisor isomorphic to $D$ and contained in $S_1$ and  $S_2$.}, and let \begin{equation}\label{expanded}\mathfrak{S}\to \fC\to C \end{equation} be the universal family of surfaces over the stack of expanded degenerations $\fC$  (see \cite{L01, L02},\cite[Sections 2.1-2.2]{LW15}). They fit into the fibered diagram $$\xymatrix{
\mathfrak S \ar[d] \ar[r]& X \ar[d]\\ \mathfrak{C} \ar[r] &C.
}$$  Following the construction of Li and Wu \cite{LW15}, one can construct the nested Hilbert scheme of points, denoted by $\mathfrak{S}^{[n_1\ge n_2]}$ on the fibers of $\mathfrak{S}$. The Hilbert scheme $\mathfrak{S}^{[n_1\ge n_2]}$ is a proper Deligne-Mumford stack over $\fC$ and its structure morphism factors through the substack $\fC^{[\n]}\subset \fC$ corresponding to the numerical data $\n$ (see \cite[Secttion 2.5]{LW15} for the construction of these substacks). $\fC^{[\n]}$ is a smooth Artin stack of dimension 1. Let $\fC^{[\n]}_0\subset \fC^{[\n]}$ be the substack corresponding to $0\in C$. A non-special fiber of $\mathfrak{S}^{[n_1\ge n_2]}$ is isomorphic to $\S{n_1\ge n_2}$, whereas the special fiber of $\mathfrak{S}^{[n_1\ge n_2]}$, denoted by  $S^{[n_1\ge n_2]}_0$, can be written as the (non-disjoint) union  \begin{equation} \label{central}S^{[n_1\ge n_2]}_0=\bigcup_{\tiny \begin{array}{c} \n=\n'+\n'' \end{array}} (S_1/D)^{[n'_1\ge n'_2]}\times (S_2/D)^{[n''_1\ge n''_2]}, \end{equation} where $\n'=(n'_1,n'_2)$ and $\n''=(n''_1,n''_2)$ with $n'_1 \ge n'_2  $ and $n''_1 \ge n''_2  $. Each component $$S^{[\n]}_{0,\n',\n''}:=(S_1/D)^{[n'_1\ge n'_2]}\times (S_2/D)^{[n''_1\ge n''_2]}$$ is the pull-back of a divisor $\fC^{[\n]}_{\n',\n''}\subset \fC^{[\n]}$. Let $\mathfrak{L}^{[\n]}_{\n',\n''}$ be the corresponding line bundle. We then have $$\bigotimes_{\n=\n'+\n''}\mathfrak{L}^{[\n]}_{\n',\n''}\cong \mathfrak{L}_0$$ where $\mathfrak{L}_0$ is the line bundle associated to the pull back of the divisor $\{0\}\subset C$.

 We denote the universal objects over $\mathfrak{S}\times_{\fC^{[\n]}} \mathfrak{S}^{[n_1\ge n_2]}$ by $$0\neq \Phi: \fI{n_1}\to \fI{n_{2}}.$$  The restriction of $\Phi$ to the component $\mathfrak{S}\times_ {\fC^{[\n]}}\left( (S_1/D)^{[n'_1\ge n'_2]}\times (S_2/D)^{[n''_1\ge n''_2]}\right)$ is identified with the pair of universal maps $$(\i{n'_1}\hookrightarrow \i{n'_2}, \i{n''_1}\hookrightarrow \i{n''_2}).$$ We denote by $\pi$ the projection to the second factor of $\mathfrak{S}\times_{\fC^{[\n]}} \mathfrak{S}^{[n_1\ge n_2]}$, and by $p$ the projection to its first factor followed by the natural morphism to the total space of the good degeneration of $S$ over $C$.
Again by the method of Section \ref{2-step} and \cite{MPT10, LW15},  one can construct a relative perfect obstruction theory $\frak{F}^\bullet_{\rel}\to \LL_{ \mathfrak{S}^{[n_1\ge n_2]}/\fC^{[\n]}}$ with the relative virtual tangent bundle:
\begin{align*}&\frak{F}_{\rel}^{\bullet \vee}=\\&\cone \bigg( \left[\dR\hom_\pi \big(\fI{n_1},\fI{n_1}\big)\oplus \dR\hom_\pi\big(\fI{n_2},\fI{n_2}\big)\right]_0 \to \dR \hom_\pi \big(\fI{n_1}, \fI{n_2}\big)\bigg).\end{align*} Since $\fC^{[\n]}$ is a smooth Artin stack ($h^{-1}(\LL_{\fC^{[\n]}})=0$ but $h^{1}(\LL_{\fC^{[\n]}})\neq 0$) there exists an absolute perfect obstruction theory $\frak{F}^{\bullet}\to \LL_{ \mathfrak{S}^{[n_1\ge n_2]}}$ associated to $\frak{F}_{\rel}^{\bullet}$ (see for example the argument after diagram (49) in \cite{MPT10}).
The restriction of $\frak{F}_{\rel}^{\bullet}$ to  $S_0^{[\n]}$ and its components $S_{0,\n',\n''}^{[\n]}$ induce relative perfect obstruction theories  $$\frak{F}_{0}^{\bullet}\to \LL_{S_0^{[\n]}/\fC^{[\n]}_0} \qquad \text{and} \qquad \frak{F}_{0,\n',\n''}^{\bullet}\to \LL_{S_{0,\n',\n''}^{[\n]}/\fC^{[\n]}_{0,\n',\n''}},$$ respectively. As in \cite{MPT10},  they satisfy the following compatibilities:
\begin{equation} \label{compatibilities} \frak{F}^{\bullet}|_{S_0^{[\n]}}\to \frak{F}_{0}^{\bullet}\to \mathfrak{L}^\vee_0[1],\quad \quad \frak{F}^{\bullet}|_{S_{0,\n',\n''}^{[\n]}}\to \frak{F}_{0,\n',\n''}^{\bullet}\to \mathfrak{L}_{\n',\n''}^{[\n] \vee}[1],\end{equation}  where each sequence is an exact triangle.

A decomposition $S_0[k_1,k_2]:=S_1[k_1]\cup_D S_2[k_2]$ yields the natural exact sequence \begin{equation*}  0\to \O_{S_0[k_1,k_2]}\to \O_{S_1[k_1]}\oplus \O_{S_2[k_2]}\to \O_D\to 0.\end{equation*} Suppose that $I_1\subseteq I_2$ is a nested pairs of relative ideal sheaves on $S_0[k_1,k_2]$, and let $I_i':=I_i|_{S_1[k_1]}$ and $I''_i:=I_i|_{S_2[k_2]}$. Tensoring the short exact sequence above with the perfect complexes $\dR\hom(I_1,I_1)$, $\dR\hom(I_2,I_2)$, and $\dR\hom(I_1,I_2)$ and applying $\dR\Gamma$ we get the commutative diagram
$$\xymatrix{
\oplus_{i=1}^2 \dR\Hom(I_i,I_i) \ar[r] \ar[d] &\oplus_{i=1}^2 \dR\Hom(I'_i,I'_i)\oplus\oplus_{i=1}^2 \dR\Hom(I''_i,I''_i)\ar[r] \ar[d]& \dR \Gamma \O_D \oplus \dR \Gamma \O_D \ar[d]^-{[-1\; 1]} \\ 
\dR\hom(I_1,I_2) \ar[r] & \dR\Hom(I'_1,I'_2)\oplus \dR\Hom(I''_1,I''_2)\ar[r] & \dR \Gamma \O_D
}
$$ where each row is an exact triangle and the the first two vertical maps are induced from the natural inclusions $I_1\subseteq I_2$, $I'_1\subseteq I'_2$ and $I''_1\subseteq I''_2$ as in Section \ref{2-step}, and in the thrid column we have used the relativity condition of ideal sheaves i.e. $I_i|_D=\O_D$.
As before the vertical maps factor through the trace free parts and hence, using the natural exact triangle
$$\dR\Gamma \O_S\to \dR\Gamma \O_{S_1}\oplus \dR\Gamma \O_{S_2}\to \dR\Gamma \O_{D} ,$$
this induces the following commutative diagram of the exact triangles
$$\xymatrix{
[\oplus_{i=1}^2 \dR\Hom(I_i,I_i)]_0 \ar[r] \ar[d] & [\oplus_{i=1}^2 \dR\Hom(I'_i,I'_i)]_0\oplus[\oplus_{i=1}^2 \dR\Hom(I''_i,I''_i)]_0\ar[r] \ar[d]&  \dR \Gamma \O_D \ar@{=}[d]\\ 
\dR\hom(I_1,I_2) \ar[r] & \dR\Hom(I'_1,I'_2)\oplus \dR\Hom(I''_1,I''_2)\ar[r] & \dR \Gamma \O_D.
}
$$
Taking the cones we get the isomorphism
\begin{align*}&\cone \bigg( \left[\dR \Hom \left(I_1,I_1\right)\oplus \dR\Hom \left(I_2 ,I_2\right)\right]_0 \to \dR \Hom\left(I_1, I_2\right)\bigg) \cong \\&\quad 
\cone \bigg( \left[\dR \Hom \left(I'_1,I'_1\right)\oplus \dR\Hom \left(I'_2 ,I'_2\right)\right]_0 \to \dR \Hom\left(I'_1, I'_2\right)\bigg)\bigoplus\\ &
\quad \quad \cone \bigg( \left[\dR \Hom \left(I''_1,I''_1\right)\oplus \dR\Hom \left(I''_2 ,I''_2\right)\right]_0 \to \dR \Hom\left(I''_1, I''_2\right)\bigg).
\end{align*}

One of the upshots is that following the construction of \cite{MPT10, LW15}, we are led by the isomorphism above to the following degeneration formula for the virtual integration over $\S{n_1\ge n_2}$ (\cite[Thm. 16]{ MPT10}, \cite[Prop. 6.5, Thm. 6.6]{LW15}). This is done by using the compatibilities \eqref{compatibilities} and relating the relative prefect obstruction theory $\mathfrak{F}_{\rel}^\bullet$ to the absolute perfect obstruction theories  $\Fb$ (given in Proposition \ref{abs}) and $\mathcal{F}^\bullet_{\rel}$ (given by \eqref{obrelative}):

\begin{prop} \label{degen}  Let $\alpha$ be a cohomology class in the total space of $\mathfrak{S}^{[n_1\ge n_2]}$, then, 
$$\int_{[\S{n_1\ge n_2}]^{\vir}}\alpha=\sum_{\n=\n'+\n''}\left( \int_{[(S_1/D)^{[n'_1\ge n'_2]}]^{\vir}}\alpha\right)\cdot\left( \int_{[(S_2/D)^{[n'_1\ge n'_2]}]^{\vir}}\alpha \right).$$
\end{prop} \qed

\begin{rmk} \label{hilbpts}
In Proposition \ref{degen}, if $n=n_1=n_2$ then $\mathfrak{S}^{[n_1\ge n_2]}\cong \mathfrak{S}^{[n]}$ constructed by \cite{LW15}, and by the same argument as in proof of Theorem \ref{thm1.2} part 1, one can recover the usual degeneration formula for the Hilbert schemes of points used in \cite{T12, LT14, GS16}:
\begin{equation} \label{usualdegen}
\int_{\S{n}}\alpha=\sum_{n=n'+n''}\left( \int_{(S_1/D)^{[n']}}\alpha\right)\cdot\left( \int_{[(S_2/D)^{[n'']}}\alpha \right).
\end{equation}
\end{rmk}

\subsection{Double point relation} \label{2ble} 
Let $(S,D)$ be a pair of nonsingular projective surface and a nonsingular effective divisor as in Section \ref{sec:relative}, and let $M$ be line bundle on $S$. Also, recall the definitions of $$\pi:\mathcal{S}\times_{\mathfrak{A}^{[\n]}_\diamond}\SD{\n}\to \SD{\n},\quad \quad p:\mathcal{S}\times_{\mathfrak{A}^{[\n]}_\diamond}\SD{\n}\to S.$$ Let $\D\subset \mathcal{S}$ be the proper transform of $D\subset S$ via $p$. Note that we use $\pi$ and $p$ for the similar natural morphisms from $\mathfrak{S}\times_{\fC^{[\n]}} \mathfrak{S}^{[n_1\ge n_2]}$ as well.
\begin{defi} \label{Kgrpelt} Define the following element in  $K(\SD{n_1\ge n_2})$ of rank $n_1+n_2$: $$\kk{n_1\ge n_2}_M:=\left[\dR\pi_*p^*M \right]-\left[\dR\hom_\pi(\i{n_1},\i{n_2}\otimes p^*M)\right].$$  
Define the following generating series:
$$Z_{\nest}(S/D,M):=\sum_{n_1\ge n_2\ge 0}q_1^{n_1}q_2^{n_2}\int_{[\SD{n_1\ge n_2}]^{\vir}}c(\kk{n_1\ge n_2}_M).$$
If  $D=0$ we drop it from the notation.
\end{defi}


\begin{lem}\label{satis}  Given a good degeneration $S  \rightsquigarrow S_0:= S_1\cup_D S_2$, and a choice of degeneration of line bundles  $$\pic(S) \ni M \rightsquigarrow  M_i \in \pic(S_i) \quad i=1, 2,$$ we get the degeneration of the class $c(\sK^{[n_1\ge n_2]}_M)$ whose restriction to the component $$(S_1/D)^{[n'_1\ge n'_2]}\times (S_2/D)^{[n''_1\ge n''_2]}$$ of the central fiber of  $\mathfrak{S}^{[n_1\ge n_2]}$ is $c(\kk{n'_1\ge n'_2}_{M_1}) \boxtimes c(\kk{n''_1\ge n''_2}_{M_2}).$
\end{lem}


\begin{proof}
Let $\mM$ be the line bundle over the total space of the good degeneration of $S$ that gives the degeneration of $M$ as in the lemma. The derived pullbacks of the perfect complexes $\dR\hom_\pi(\fI{n_1},\fI{n_2}\otimes p^*\mM)$ and $\dR\pi_* p^*\mM$ to the component $(S_1/D)^{[n'_1\ge n'_2]}\times (S_2/D)^{[n''_1\ge n''_2]}$ fits in the exact triangles 
\begin{align*}&\dR\mathcal{H}om_\pi(\fI{n_1},\fI{n_2}\otimes p^*\mM)\\&\to \dR\mathcal{H}om_\pi(\i{n'_1},\i{n'_2}\otimes p^*M_1)\oplus \dR\mathcal{H}om_\pi(\i{n''_1},\i{n''_2}\otimes p^*M_2)\\&\to \dR\mathcal{H}om_\pi(\O_\D,\O_\D \otimes p^*M|_D)\cong \dR\pi_* p^*M|_D,\end{align*} and 
$\dR\pi_* p^*\mM \to \oplus_{i=1}^2 \dR\pi_* p^*M_i\to \dR\pi_* p^*M|_D.$
Now taking the difference of the $K$-group classes from the exact triangles above, and applying the total Chern class, we conclude that $c(\kk{n_1\ge n_2}_{M})$ degenerates to a class whose restriction to the component $${(S_1/D)^{[n'_1\ge n'_2]}\times (S_2/D)^{[n''_1\ge n''_2]}}$$ is  $c(\kk{n'_1\ge n'_2}_{M_1} \boxplus \kk{n''_1\ge n''_2}_{M_2})$. 

\end{proof}

A direct corollary of Proposition \ref{degen}  and Lemma \ref{satis} is 
\begin{prop} \label{gooddegen} Given a good degeneration $S  \rightsquigarrow S_0:= S_1\cup_D S_2$, and a choice of degeneration of line bundles  $$\pic(S) \ni M \rightsquigarrow  M_i \in \pic(S_i) \quad i=1, 2,$$ we have 
\begin{equation} \label{degenformul} Z_{\nest}(S,M)=Z_{\nest}(S_1/D,M_1)\cdot Z_{\nest}(S_2/D,M_2).\end{equation}
\end{prop} \qed

In the situation of Proposition \ref{gooddegen}, Let $\PP$ be either of the projective bundles $\PP(\O_D+N_{S_1/D})\cong \PP(\O_D+N_{S_2/D})$, and let $M_{\PP}$ be the pullback of $M|_D$ to $\PP$. Applying Proposition \ref{gooddegen} to the degeneration to the normal cone of $D \subset S_i$ gives 
\begin{equation} \label{degenformul1} Z_{\nest}(S_i,M_i)=Z_{\nest}(S_i/D,M_i)\cdot Z_{\nest}(\PP/D,M_{\PP}).\end{equation} Similarly, the degeneration to the normal cone of $D \subset \PP$ gives 
\begin{equation} \label{degenformul2} Z_{\nest}(\PP,M_{\PP})=Z_{\nest}(\PP/D,M_{\PP})\cdot Z_{\nest}(\PP/D,M_{\PP}).\end{equation}

Let $\mathcal{M}_{2,1}(\C)^+$ be the group completion of the set of isomorphism classes of the pairs $(S,M)$, where $S$ is a smooth projective surface over $\C$ and $M$ is a line bundle on $S$ (see \cite[Definition 3]{LP12} ).

\begin{cor} \label{descent}
$Z_{\nest}(-,-)$ satisfies the relation\footnote{This relation is the analog of the relation (0.10) in \cite{LP09}.} \begin{equation} \label{degenformul3} Z_{\nest}(S,M)\cdot Z_{\nest}(S_1,M_1)^{-1}\cdot Z_{\nest}(S_2,M_2)^{-1}\cdot Z_{\nest}(\PP,M_{\PP})=1\end{equation}
 and hence it respects the double point relations in  $\mathcal{M}_{2,1}(\C)^+$. In other words, $Z_{\nest}(-,-)$ descends to a homomorphism 
$$Z_{\nest}(-,-): \omega_{2,1}(\C)\otimes_{\mathbb{Z}} \mathbb{Q}\to \mathbb{Q}[[q_1,q_2]]^*, $$ where $\omega_{2,1}(\C)$ is  the double point cobordism theory for line bundles on surfaces obtained by taking the quotient of $\mathcal{M}_{2,1}(\C)^+$ by all the double point relations.
\end{cor}
\begin{proof}
Relation \eqref{degenformul3} follows immediately from relations \eqref{degenformul}-\eqref{degenformul2}.
\end{proof}
It is known that $\omega_{2,1}(\C)$ is generated by the following classes (see \cite{T12, LP12})
\begin{equation} \label{generators} [\PP^2,\O],\quad [\PP^2,\O(1)], \quad [\PP^1\times \PP^1,\O], \quad [\PP^1\times \PP^1,\O(1,0)].\end{equation}
Let $B_1,B_2,B_3,B_4$ be $Z_{\nest}(S,M)$ where $(S,M)$ is one of the pairs above respectively from left to right.
Define $$A_1:=B_1^{-1}B_2B_3^{3/2}B_4^{-3/2},\ A_2:=B_3^{1/2}B_4^{-1/2},\ A_3:=B_1^{-1/3}B_3^{-1/4},\ A_4:=B_1^{-2/3}B_3^{3/4}.$$

\begin{prop}\label{4series}
Let $S$ be a nonsingular projective surface and $M$ be a line bundle on $S$. Let $A_1, A_2, A_3, A_4 \in \mathbb{Q}[[q_1,q_2]]^*$ be defined as above (independent of $S$) then
$$Z_{\nest}(S,M)=A_1^{M^2}A_2^{M\cdot K_S}A_3^{K^2_S}A_4^{c_2(S)}.$$
\end{prop}
\begin{proof} By basechange, $$c(\sE^{n_1,n_2}_M)|_{\S{n_1\ge n_2}}=c(\sK_M^{[n_1\ge n_2]}),$$ 
and so by Proposition \ref{nestprodfano}, we have 
$$Z_{\nest}(S,M)=\sum_{n_1\ge n_2\ge 0}q_1^{n_1}q_2^{n_2}\int_{\S{n_1}\times \S{n_2}}c_{n_1+n_2}(\sE^{n_1,n_2})\cup c(\sE_M^{n_1,n_2}).$$ if $(S,M)$ is one of the generators \eqref{generators}. For a general $(S,M)$ as in the proposition we can express the class $[S,M]\in \omega_{2,1}(\C)$ as a linear combination of the generators \eqref{generators} \cite[Proposition 4.1]{T12}. The result then follows by applying the homomorphism $Z_{\nest}(-,-)$ of Corollary \ref{descent} and then rearranging the factors as in the proof of \cite[Proposition 4.1]{T12}.
\end{proof}

\begin{cor}\label {univpoly}
Let $S$ be a nonsingular projective surface and $M$ be a line bundle on $S$. Then the integral $$\int_{[\S{n_1\ge n_2}]^{\vir}}c(\kk{n_1\ge n_2}_M)$$ can be written as a degree $n_1+n_2$ universal polynomial in $M^2, M\cdot K_S, K_S^2, c_2(S)$.
\end{cor}
\begin{proof} The integral in the proposition is the coefficient of $q_1^{n_1}q_2^{n_2}$ in $Z_{\nest}(S,M)$. The result follows after expanding the right hand side of the formula in Proposition \ref{4series} and extracting the coefficient of $q_1^{n_1}q_2^{n_2}$.
\end{proof}

\begin{cor} \label{cor:znzp} For any nonsingular projective surface $S$ and $M\in \pic(S)$ let $\cP:=c(\kk{n_1\ge n_2}_M)$. Then
$$\sN_S(n_1,n_2,0;\cP)=\int_{\S{n_1}\times \S{n_2}} c_{n_1+n_2}(\sE^{n_1,n_2})\cup c(\sE^{n_1,n_2}_M).$$
\end{cor} 

\begin{proof}
By Corollary \ref{univpoly} we know that the LHS of the statement above  i.e. 
$$\sN_S(n_1,n_2,0;\cP)=\int_{[\S{n_1\ge n_2}]^{\vir}} c(\kk{n_1\ge n_2}_M)$$ is a universal polynomial $P_1$ in $M^2, M\cdot K_S, K_S^2, c_2(S)$. On the other hand, using Grothendieck-Riemann-Roch formula and the induction scheme of Ellingsrud, G\"{o}ttsche, and Lehn (see \cite[Sections 3, 4]{EGL99} and \cite[Section 3]{CO12}),
we can express the RHS of the corollary in terms of a universal polynomial $P_2$ in $M^2, M\cdot K_S, K_S^2, c_2(S)$. But since the equality in the corollary holds for any $(S,M)$ in which $S$ is toric (by Proposition \ref{nestprodfano}), we conclude that $P_1=P_2$, and the result follows.

\end{proof}

This corollary allows us to write integrlas against $[\S{n_1 \ge n_2}]^{\vir}$ in terms of integrations over the product of Hilbert schemes of points. As said in the introduction the following generalization of this corollary is used in \cite{GSY17b} in the context of reduced localized DT invariants. In fact the only property of the integrand $c(\kk{n_1\ge n_2}_M)$ that was needed in the argument leading to Corollary \ref{cor:znzp} was its decomposition property under good degenerations of $S$ as stated in Lemma \ref{satis}. We can therefore prove similar identities as in Corollary \ref{cor:znzp} for other integrands with such a decomposition property. For example, by a same proof as Lemma \ref{satis} one can see that the Chern class of the twisted tangent bundle (Definition \ref{virbdl}), $c(\sT^M_{\S{n_i}})$, also has this property, and so does any product/quotient of these Chern classes. In particular, we can extend  Corollary \ref{cor:znzp} to the following more general statement: 
\begin{prop}\label{genznzp} 
 Let $L_1,\dots, L_s$, $L'_1,\dots, L'_{s'}$, $M_1,\dots, M_{t}$, be some line bundles on the nonsingular projective surface $S$, and $l_1,\dots, l_s$, $l'_1,\dots, l'_{s'}$, $m_1,\dots, m_{t}$ be finite sequences of $\pm 1$. Define
$$\cP:=\prod_{i=1}^s c(\sT^{L_i}_{\S{n_1}})^{l_i}\cup\prod_{i=1}^{s'} c(\sT^{L'_i}_{\S{n_2}})^{l'_i}\cup \prod_{i=1}^t c(\kk{n_1\ge n_2}_{M_i})^{m_i}.$$
Then,
\begin{align*}&\sN_S(n_1,n_2,0;\cP)=\\&\int_{\S{n_1}\times \S{n_2}} c_{n_1+n_2}(\sE^{n_1,n_2})\cup \prod_{i=1}^s c(\sT^{L_i}_{\S{n_1}})^{l_i}\cup\prod_{i=1}^{s'} c(\sT^{L'_i}_{\S{n_2}})^{l'_i}\cup \prod_{i=1}^t c(\sE^{n_1, n_2}_{M_i})^{m_i}.\end{align*}
\end{prop} \qed

\section{Vertex operator formulas and proof of Theorem \ref{thm4}}  \label{sec:co} Let $(S,M)$ be a pair of a projective nonsingular surface $S$ and a line bundles $M$ on $S$. Let $\sF=\oplus_n H^*(\S{n},\mathbb{Q})$.
Carlsson and Okounkov defined the operator $W(M_1)$ in $\operatorname{End}(\sF)[[z_1,z_1^{-1}]]$ by 
$$\langle W(M_1)\eta_1,\eta_2\rangle :=z_1^{n_2-n_1}\int_{\S{n_1}\times \S{n_2}}p_1^*\eta_1 \cup p_2^* \eta_2 \cup c_{n_1+n_2}(\sE^{n_1,n_2}_{M_1}), $$ where
$\langle-,-\rangle$ is the Poincar\'e pairing, $\eta_i \in H^*(\S{n_i},\mathbb{Q})$, $p_i$ is the projection to the $i$-th factor of  $\S{n_1}\times \S{n_2}$, and $\sE^{n_1,n_2}_{M_1}\in K(\S{n_1}\times \S{n_2})$ is as in Definition \ref{virbdl}. In other words, using \cite[Definition 16.1.2]{F13}, $W(M_1)$ is the operator associated to the family of correspondences $$c_{n_1+n_2}(\sE^{n_1,n_2}_{M_1}):\S{n_1} \vdash \S{n_2}\quad  n_1, n_2\ge 0.$$
If $M_2$ is another line bundle on $S$, we define $$W(M_1,M_2)(z_1,z_2):=W(M_2)(z_2)\circ W(M_1)(z_1).$$
By \cite[Proposition 16.1.2]{F13}, 
\begin{align*}
&\langle W(M_1,M_2)\eta_1,\eta_3\rangle=\\&\sum_{n_2} z_1^{n_2-n_1}z_2^{n_3-n_2}\int_{\S{n_1}\times \S{n_2}\times \S{n_3}} p_1^*\eta_1 \cup p_3^* \eta_3 \cup c_{n_1+n_2}(\sE^{n_1,n_2}_{M_1})\cup c_{n_2+n_3}(\sE^{n_2,n_3}_{M_2}),\end{align*} where
$\eta_i \in H^*(\S{n_i},\mathbb{Q})$, and $p_i$ is the projection to the $i$-th factor of  $\S{n_1}\times \S{n_2}\times \S{n_3}$.

Carlsson and Okounkov found an explicit formula for $W(-)$ in terms of vertex operators. Let $\alpha_{\pm}(-)$ denote Nakajima's annihilation/creation operators.
\begin{thm} (Carlsson-Okounkov \cite{CO12}) \label{CO}
$$W(M_1)=\Gamma_{-}(-M_1,-z_1)\circ \Gamma_{+}(-M_1^D,z_1),$$ where $$ \Gamma_{\pm}(M_1,z_1):=\exp\left(\sum_{n>0}\frac{ z_1^{\mp n}}{n}\alpha_{\pm n}(M_1)\right).$$
\end{thm}\qed

 Note that the operators $\Gamma_{\pm}$ satisfy the commutation relations $[\Gamma_{\pm}, \Gamma_{\pm}]=0$, and moreover,  $$\Gamma_{+}(M_2,z_2)\circ \Gamma_{-}(M_1,z_1)= (1+\frac{z_1}{z_2})^{\langle M_1,M_2 \rangle} \Gamma_{-}(M_1,z_1)\circ \Gamma_{+}(M_2,z_2).$$ Let $\c:=(1-\frac{z_1}{z_2})^{\langle M_1,M_2^D \rangle}$. Using these properties, we can write

\begin{align*}
W(M_1,M_2)&=
\Gamma_{-}(-M_2,-z_2)\circ \Gamma_{+}(-M_2^D,z_2)\circ \Gamma_{-}(-M_1,-z_1)\circ \Gamma_{+}(-M_1^D,z_1)\notag\\
&=\c \;\Gamma_{-}(-M_2,-z_2)\circ \Gamma_{-}(-M_1,-z_1)\circ \Gamma_{+}(-M_2^D, z_2)\circ \Gamma_{+}(-M_1^D, z_1).
\end{align*}
Let $\sN$ be the number-of-points operator: $\sN |_{\sF_n}=n\id$. It satisfies $$q^{\sN}\circ \Gamma_{-}(M_i,z_i)=\Gamma_{-}(M_i,qz_i)\circ q^{\sN}.$$
Starting with $\str\left(q^{\sN} \circ W(M_1,M_2)\right)$ and using the commutation relation of the super-trace $\str(A\circ B)=\str(B\circ A)$, we obtain
\begin{align*}
&\c \str\left( q^{\sN} \circ \Gamma_{-}(-M_2,-z_2)\circ \Gamma_{-}(-M_1,-z_1)\circ \Gamma_{+}(-M_2^D, z_2)\circ \Gamma_{+}(-M_1^D, z_1)\right)=\notag\\
&
\c \str \left(\Gamma_{-}(-M_2,-z_2q)\circ \Gamma_{-}(-M_1,-z_1q)\circ  q^{\sN}\circ \Gamma_{+}(-M_2^D,z_2)\circ \Gamma_{+}(-M_1^D,z_1)\right)=\notag\\
&
\c \str\left( q^{\sN}\circ \Gamma_{+}(-M_2^D,z_2)\circ \Gamma_{+}(-M_1^D,z_1) \circ \Gamma_{-}(-M_2,-z_2q)\circ \Gamma_{-}(-M_1,-z_1q)\right)=\notag\\
&
(1- \frac{z_1 q}{z_2})^{ \langle M_1, M_2^D \rangle }(1- \frac{z_2q}{z_1})^{\langle M_1^D, M_2\rangle }(1- q)^{\langle M_1^D, M_1\rangle +\langle M_2^D, M_2\rangle }\notag\\& \c \str\left( q^{\sN}\circ \Gamma_{-}(-M_2,-z_2q)\circ \Gamma_{-}(-M_1,-z_1q)\circ \Gamma_{+}(-M_2^D,z_2)\circ \Gamma_{+}(-M_1^D,z_1)\right).
\end{align*}
Iterating this process, we get
\begin{align*}
&\str \left( q^{\sN}\circ W(M_1,M_2)\right)=\\
& \prod_{n>0 }(1- \frac{z_1q^n}{z_2})^{\langle M_1, M_2^D\rangle } (1- \frac{z_2q^n}{z_1})^{\langle M_1^D, M_2\rangle }(1- q^{n})^{\langle M_1^D, M_1\rangle+\langle M_2^D, M_2\rangle }\notag\\
&
\c \str\left(q^{\sN}\circ \Gamma_{+}(-M_2^D,z_2)\circ\Gamma_{+}(-M_1^D,z_1)\right)=
\notag\\
& \prod_{n\ge 0}(1- \frac{z_1q^n}{z_2})^{\langle M_1, M_2^D\rangle}(1- q^{n})^{\langle M_1^D, M_1\rangle+\langle M_2^D, M_2\rangle-e(S)}\prod_{n>0}(1- \frac{z_2q^n}{z_1})^{\langle M_1^D, M_2\rangle},
\end{align*}  where for the last equality, we have used the fact that $\Gamma_+$ is a lower triangular operator, and G\"ottsche's formula $$\sum_{n\ge 0} e(\S{n})q^n=\prod_{n\ge 0}(1-q^n)^{-e(S)}.$$
Define $$q_1:=qz_1^{-2}, \quad q_2:=z_1^{2},$$ then we have shown
\begin{align} \label{str}
&
\str\left(q^{\sN}\circ W(M_1,M_2)(z_1, 1/z_1)\right)=\\& \prod_{n\ge 0}(1- q_2^{n+1}q_1^n)^{\langle M_1, M_2^D\rangle}(1- (q_1q_2)^{n})^{\langle M_1^D, M_1\rangle+\langle M_2^D, M_2 \rangle-e(S)}\prod_{n>0}(1- q_2^{n-1}q_1^n)^{\langle M_1^D, M_2 \rangle}.\notag
\end{align}
On the other hand, by \cite[Example 16.1.3]{F13},
\begin{align} \label{str1}
\str\left(q^{\sN}\circ W(M_1,M_2)(z_1, 1/z_1)\right)&=q^{n_1}z_1^{2(n_2-n_1)} \int_{\S{n_1}\times \S{n_2}}c_{n_1+n_2}(\sE^{n_1,n_2}_{M_1})c_{n_1+n_2}(\sE^{n_2,n_1}_{M_2})\\&\notag
=(-1)^{n_1+n_2}q_1^{n_1}q_2^{n_2}\int_{\S{n_1}\times \S{n_2}}c_{n_1+n_2}(\sE^{n_1,n_2}_{M_1})c_{n_1+n_2}(\sE^{n_1,n_2}_{M_2^D}),
\end{align}
where the last equality is because of Grothendieck-Verdier duality.

\begin{notn}
If $Z=\sum_{r_1,r_2\ge 0} a_{r_1,r_2}q_1^{r_1}q_2^{r_2}$ is a formal series, we define $$Z\, \left[q_{1}^{n_1}q_2^{n_2}\right]:=a_{n_1,n_2}.$$
\end{notn}

The following proposition completes the proof of Theorem \ref{thm4}:
\begin{prop} \label{coformul} Let $(S,M)$ be a pair of a projective nonsingular surface $S$ and a line bundles $M$ on $S$, then, 
\begin{align*}\int_{[\S{n_1\ge n_2}]^{\vir}} &c_{n_1+n_2}(\sK^{[n_1\ge n_2]}_M)=\\&(-1)^{n_1+n_2}
\prod_{n> 0}(1- q_2^{n-1}q_1^n)^{\langle K_S, M^D \rangle}(1- (q_1q_2)^{n})^{\langle M^D, M \rangle-e(S)}\left[q_{1}^{n_1}q_2^{n_2}\right].
\end{align*}
\end{prop}
\begin{proof} 
By \eqref{str1},
\begin{align*}\int_{\S{n_1}\times \S{n_2}}& c_{n_1+n_2}(\sE^{n_1, n_2})\cup c_{n_1+n_2}(\sE^{n_1, n_2}_M)=\\&(-1)^{n_1+n_2}\str\left(q^{\sN}\circ W(\O_S,M^D)(z_1,1/z_1)\right)\left[q_{1}^{n_1}q_2^{n_2}\right].\end{align*}
The result now follows immediately from \eqref{str} and Corollary \ref{cor:znzp}.
\end{proof}

\noindent {\tt{amingh@math.umd.edu}},\quad 
\noindent {\tt{University of Maryland}} \\
\noindent {\tt{College Park, MD 20742-4015, USA}} \\\\

\noindent{\tt{artan@cmsa.fas.harvard.edu, Center for Mathematical Sciences and\\ Applications, Harvard University, Department of Mathematics, 20 Garden Street, Room 207, Cambridge, MA, 02139}}\\\\
\noindent{\tt{artan@cmsa.fas.harvard.edu, Centre for Quantum Geometry of Moduli Spaces, Aarhus University, Department of Mathematics
Ny Munkegade 118, building 1530, 319, 8000 Aarhus C, Denmark}}\\\\
\noindent{\tt{artan@cmsa.fas.harvard.edu, National Research University Higher School of Economics, Russian Federation, Laboratory of Mirror Symmetry, NRU HSE, 6 Usacheva str.,Moscow, Russia, 119048}}\\\\
\noindent{\tt{yau@math.harvard.edu}}
\noindent{\tt{Department of Mathematics, Harvard University, Cambridge, MA 02138, USA }}


\begin{thebibliography}{10}
\bibitem[BF97]{BF97} Kai Behrend,  and Barbara Fantechi. ``The intrinsic normal cone.'' \emph{Inventiones Mathematicae} 128 (1997): 45--88.

\bibitem[BFl03]{BFl03} Ragnar-Olaf Buchweitz,  and Hubert Flenner. ``A semiregularity map for modules and applications to deformations.'' \emph{Compositio Mathematica} 137 (2003): 135--210.

\bibitem[CO12]{CO12} Erik Carlsson,  and Andrei Okounkov. ``Exts and vertex operators.'' \emph{Duke Math. J.}, 161 (2012): 1797--1815.     


\bibitem[CK13]{CK13} Huai-liang Chang,  and Young-Hoon Kiem. ``Poincar\'e invariants are Seiberg--Witten invariants." \emph{Geometry \& Topology} 17 (2013): 1149--1163.

\bibitem[C98]{C98} Jan Cheah. ``Cellular decompositions for nested Hilbert schemes of points.'' \emph{Pacific J. of Math.} , 183 (1998): 39--90.


\bibitem[DKO07]{DKO07} Markus D\"urr,  and Alexandre Kabanov, and Christian Okonek. ``Poincar\'e  invariants.'' \emph{Topology}, 46 (2007): 225--294.

\bibitem[ES87]{ES87} Geir Ellingsrud,  and Stein Arild Str{{\o}}mme. ``On the homology of the Hilbert scheme of points in the plane.'' \emph{Inventiones Mathematicae} 87(1987): 343--352.

\bibitem[EGL99]{EGL99} Geir Ellingsrud,  and Lothar G\"{o}ttsche, and Manfred Lehn. ``On the cobordism class of the Hilbert scheme of a surface.'' \emph{J. of Algebraic Geometry}, 10 (2001): 81--100.

\bibitem[F]{F69} [F69] J. Fogarty.`` Truncated Hubert Functors". \emph{Reine u. Ang. Math.} 234 (1969):1--65.


\bibitem[F13]{F13} William Fulton. ``Intersection theory'' Vol. 2. \emph{Springer Science $\&$ Business Media}, (2013).

\bibitem[G11]{G11}W. D. Gillam. ``Deformation of quotients on a product". arXiv:1103.5482, (2011).


\bibitem[G90]{G90} Lothar G\"{o}ttsche. ``The Betti numbers of the Hilbert scheme of points on a smooth projective surface'', \emph{Math Ann.},  286 (1990): 193--207.        


\bibitem[GP99]{GP99} Tom Graber,  and Rahul Pandharipande. ``Localization of virtual classes.'' \emph{Inventiones Mathematicae} 135 (1999): 487--518.

\bibitem[GS16]{GS16} Amin Gholampour,  and Artan Sheshmani. ``Intersection numbers on the relative Hilbert schemes of points on surfaces.'' \emph{Asian Journal of Mathematics}. 21(2017): 531--542.\bibitem[GS13]{GS13} Amin Gholampour,  and Artan Sheshmani. ``Donaldson-Thomas Invariants of 2-Dimensional sheaves inside threefolds and modular forms.'' arXiv preprint arXiv:1309.0050 (2013).

\bibitem[GSY17b]{GSY17b} Amin Gholampour,  and Artan Sheshmani, and Shing-Tung Yau. ``Localized Donaldson-Thomas theory of surfaces", arXiv:1701.08902 (2017).

\bibitem[GT17]{GT17} Amin Gholampour, and Richard P. Thomas. ``Degeneracy loci, virtual cycles and nested Hilbert schemes I." arXiv:1702.04128 (2017).


\bibitem[GT19]{GT19} Amin Gholampour, and Richard P. Thomas. ``Degeneracy loci, virtual cycles and nested Hilbert schemes II." arXiv:1902.04128 (2019).



\bibitem[HL10]{HL10} Daniel Huybrechts,  and Manfred Lehn. ``The geometry of moduli spaces of sheaves.'' \emph{Cambridge University Press}, (2010).

\bibitem[HT10]{HT10} Daniel Huybrechts,  and Richard P. Thomas. ``Deformation-obstruction theory for complexes via Atiyah and Kodaira-Spencer classes.'' \emph{Mathematische Annalen} 346 (2010): 545--569.

\bibitem[Ill]{Ill} L. Illusie. ``Complexe Cotangent et Deformations I.''  \emph{Lecture Notes in Math} 239.

\bibitem[JS12]{JS12} Dominic Joyce,  and Yinan Song. ``A theory of generalized Donaldson-Thomas invariants.'' \emph{American Mathematical Society}, 217 (2012) No. 1020. 

\bibitem[KL13]{KL13} Young-Hoon Kiem, and Jun Li. "Localizing virtual cycles by cosections." \emph{Journal of the American Mathematical Society}, 26 (2013): 1025--1050.

\bibitem[K90]{K90} J\'{a}nos Koll\'{a}r. ``Projectivity of complete moduli.''  \emph{Journal of Differential Geometry}, 32 (1990): 235--268.


\bibitem[KM77]{KM77} Finn Knudsen,  and David Mumford. ``The projectivity of the moduli space of stable curves I: preliminaries on ``det''  and ``Div''.'' \emph{Mathematica Scandinavica}, 39 (1977): 19--55.

\bibitem[KST11]{KST11} Martijn Kool,  Vivek Shende, and Richard P. Thomas. ``A short proof of the G\"ottsche conjecture.'' \emph{Geometry $\&$ Topology}, 15 (2011): 397--406.

\bibitem[KT14]{KT14} Martijn Kool,  and R. P. Thomas. ``Reduced classes and curve counting on surfaces I: theory.'' \emph{Algebraic Geometry}, 1 (2014): 334--383.

\bibitem[KT14b]{KT14b} M.~Kool and R.~P.~Thomas, \textit{Reduced classes and curve counting on surfaces II: calculations}, \emph{Algebraic Geometry}, 1(2014): 384--399. 



\bibitem[L99]{L99} Manfred Lehn. ``Chern classes of tautological sheaves on Hilbert schemes of points on surfaces.'' \emph{Inventiones Mathematicae}, 136 (1999): 157--207.          

\bibitem[L04]{L04} Manfred Lehn. ``Lectures on Hilbert schemes." \emph{Algebraic structures and moduli spaces}, 38 (2004): 1--30.

\bibitem[LP12]{LP12} Y-P. Lee, and Rahul Pandharipande. ``Algebraic cobordism of bundles on varieties.'' \emph{J. Eur. Math. Soc.},  DOI 10.4171/JEMS/327 (2012).


\bibitem[LP09]{LP09} Marc Levine,  and Rahul Pandharipande. ``Algebraic cobordism revisited.'' \emph{Inventiones Mathematicae}, 176 (2009): 63--130.

\bibitem[L01]{L01} Jun Li. ``Stable morphisms to singular schemes and relative stable morphisms.'' \emph{Journal of Differential Geometry} 57 (2001): 509--578.

\bibitem[L02]{L02} Jun Li. ``A degeneration formula of GW-invariants.'' \emph{Journal of Differential Geometry} 60 (2002): 199--293.

\bibitem[LT14]{LT14} Jun Li,  and Yu-jong Tzeng. ``Universal polynomials for singular curves on surfaces.'' \emph{Compositio Mathematica}, 150 (2014): 1169--1182.

\bibitem[LW15]{LW15} Jun Li and  Baosen Wu. ``Good degeneration of Quot-schemes and coherent systems.'' \emph{Communications in Analysis and Geometry}, 23 (2015): 841--921.
   

\bibitem[MNOP06]{MNOP06} Davesh Maulik, and Nikita Nekrasov, and Andrei Okounkov, and Rahul Pandharipande. ``Gromov--Witten theory and Donaldson--Thomas theory, I.'' \emph{Compositio Mathematica}, 142 (2006): 1263--1285.   
\bibitem[MNOPII]{MNOPII} Davesh Maulik,  and Nikita Nekrasov, and Andrei Okounkov, and Rahul Pandharipande. ``Gromov--Witten theory and Donaldson--Thomas theory, II.'' \emph{Compositio Mathematica}, 142 (2006): 1286--1304.   


\bibitem[MPT10]{MPT10} Davesh Maulik,  Rahul Pandharipande, and Richard P. Thomas. ``Curves on K3 surfaces and modular forms.'' \emph{Journal of Topology}, 3 (2010): 937--996.




\bibitem[N99]{N99} Hiraku Nakajima. ``Lectures on Hilbert schemes of points on surfaces.''  \emph{American Mathematical Soc.}, 18 (1999).
 



\bibitem[N12]{N12} Andrei Negut. ``Moduli of flags of sheaves and their K-theory". arXiv:1209.4242 (2012).

\bibitem[N17]{N17} Andrei Negut. ``Shuffle algebras associated to surfaces". arXiv:1703.02027 (2017). 

\bibitem[PT10]{PT10} Rahul Pandharipande,  and Richard P. Thomas. ``Stable pairs and BPS invariants.'' \emph{Journal of the American Mathematical Society}, 23 (2010): 267-297.


\bibitem[R17]{R17} J{\o}rgen Vold Rennemo. ``Universal polynomials for tautological integrals on Hilbert schemes". \emph{Geometry \& Topology}, 21(2017): 253--314.



\bibitem[S04]{S04} Siebert B. ``Virtual fundamental classes, global normal cones and Fulton's canonical classes". In Frobenius manifolds,  Vieweg+Teubner Verlag, (2004): 341--358.


\bibitem[TT17]{TT17}  Yuuji Tanaka, and Richard P. Thomas. ``Vafa-Witten invariants for projective surfaces I: stable case." arXiv:1702.08487 (2017).



\bibitem[T12]{T12} Yu-jong Tzeng. ``A proof of the G\"ottsche-Yau-Zaslow formula." \emph{Journal of Differential Geometry} 90 (2012): 439--472.

\end{thebibliography}
\end{document}